\documentclass[11pt]{article}

\usepackage[utf8]{inputenc}
\usepackage{graphicx}
\usepackage{mathtools}
\usepackage{amsmath}
\RequirePackage{amssymb}
\RequirePackage{amsthm}
\usepackage{amsmath,scalefnt}
\usepackage{cases}
\usepackage{footnote}
\usepackage{fmtcount}
\usepackage{threeparttable}
\usepackage{caption}
\usepackage{subcaption}
\usepackage{tikz}
\usetikzlibrary{positioning,chains,fit,shapes,calc}
\usetikzlibrary{arrows}
\tikzstyle{vertex}=[circle, draw, inner sep=0pt, minimum size=6pt]
\usepackage{enumerate}
\usepackage{multicol}
\usepackage{wrapfig}
\usepackage{bm} 
\usepackage[english]{babel}
\usepackage{cleveref}
\usepackage{algorithm}
\usepackage{algorithmic}
\usepackage{amsfonts}
\usepackage{graphicx}
\usepackage{lipsum}
\usepackage{xspace}
\usepackage[section]{placeins}
\usepackage{verbatim}
\usepackage{lineno}
\usepackage{tcolorbox}
\usepackage{color}
\usepackage{colortbl}
\usepackage{xcolor}

\newtheorem{definition}{Definition}

\newtheorem{lemma}{Lemma}

\newtheorem{assumption}{Assumption}
\newtheorem{proof of proposition}{Proof of Proposition}
\newtheorem{proof of claim}{Proof of Claim}
\newtheorem{proposition}{Proposition}
\newtheorem{theorem}{Theorem}

\newtheorem{corollary}{Corollary}

\title{Bregman Proximal Gradient Algorithm with Extrapolation for a class of Nonconvex Nonsmooth Minimization Problems}
\date{Preprint submitted for publication, February 14, 2019}

\author{Xiaoya Zhang\footnotemark[1]\ \footnotemark[5] ,
Roberto Barrio\footnotemark[2] \ \footnotemark[4] ,
M. Angeles Mart\'{\i}nez\footnotemark[2]\ \footnotemark[4] ,
Hao Jiang\footnotemark[3]\ \footnotemark[5]  ,
Lizhi Cheng\footnotemark[1] \ \footnotemark[5]
}

\begin{document}
\maketitle

\renewcommand{\thefootnote}{\fnsymbol{footnote}}
\footnotetext[1]{Department of Mathematics, National University of Defense Technology, Changsha, 410073, Hunan, China. (Email: \texttt{zhangxiaoya09,clzcheng@nudt.edu.cn})}
\footnotetext[2]{Departamento de Matem\'atica Aplicada and IUMA. Computational Dynamics group.
University of Zaragoza. E-50009. Spain. (Email: \texttt{rbarrio@unizar.es})}
\footnotetext[3]{College of Computer, National University of Defense Technology, Changsha, 410073, Hunan, China. (Email: \texttt{{haojiang@nudt.edu.cn}})}
\footnotetext[4]{The authors X.Z., H.J. and L.C. were supported by the National Key Research and Development Program of China (NO.2017YFB0202003), National Natural Science Foundation of Hunan (2018JJ3616).}
\footnotetext[5]{ The authors R.B. and A.M. were supported by the Spanish Research projects MTM2015-64095-P, PGC2018-096026-B-I00, and European Regional Development Fund and Diputaci\'on General de Arag\'on (E24-17R).}
\renewcommand{\thefootnote}{\arabic{footnote}}

\begin{abstract}
In this paper, we consider an accelerated method for solving nonconvex and nonsmooth minimization problems. We propose a Bregman Proximal Gradient algorithm with extrapolation(BPGe). This algorithm extends and accelerates the Bregman Proximal Gradient algorithm (BPG), which circumvents the
restrictive global Lipschitz gradient continuity assumption needed in Proximal Gradient algorithms (PG).
The BPGe algorithm has higher generality than the recently introduced Proximal Gradient algorithm with extrapolation(PGe), and besides,  due to the extrapolation step, BPGe converges faster than BPG algorithm.
Analyzing the convergence, we prove  that any limit point of the sequence generated by BPGe is a stationary point of the problem by choosing parameters properly. Besides, assuming Kurdyka-{\L}ojasiewicz property, we prove the whole sequences generated by BPGe converges to a stationary point.
Finally, to illustrate the potential of the new method BPGe, we apply it to two important practical problems that arise in many fundamental applications (and that not satisfy global Lipschitz gradient continuity assumption): Poisson linear inverse problems and quadratic inverse problems. In the tests the accelerated BPGe algorithm shows faster convergence results, giving an interesting new algorithm.\\


\end{abstract}

\pagestyle{myheadings}
\thispagestyle{plain}

\section{Introduction}\label{Sec:1}
In the last few years different numerical methods have been devised to solve large-scale minimization problems, but still the Cauchy's gradient method is at the kernel of most of the schemes (for instance, see the recent books \cite{Bertsekas2009,borwein2010convex} and it is assumed that the gradient of the objective function is globally Lipschitz continuous. This assumption is quite restrictive in some real applications, and therefore recently new families of methods have been designed in order to solve more generic cases. On this line, the remarkable paper of Bauschke, Bolte and Teboulle \cite{bauschke2016descent} introduced a new method based on the Bregman distance paradigm (BPG algorithm) able to deal with non-globally Lipschitz continuous gradient problems in the convex case, and Bolte, Sabach, Reboulle and Vaisbourd \cite{bolte2018first} extend it to the nonconvex case.

On the other hand, a lot of effort has been paid to accelerate the proximal gradient algorithm in order to reduce the number of iterations. Several techniques have been introduced, like the fast iterative shrinkage-thresholding algorithm (FISTA) proposed in \cite{beck2009fast}, the use of Nesterov's extrapolation techniques \cite{Nesterov1983,Nesterov2007}, and quite recently it has been introduced in \cite{wen2017linear} a version of the proximal gradient algorithm with extrapolation for some nonconvex nonsmooth minimization problems (but assuming that the gradient of the objective function is globally Lipschitz continuous).

The main goal of this paper is to focus on introducing a scheme, and analyzing the convergence,  that combines the power of the method developed in
\cite{bolte2018first} able to solve non-globally Lipschitz continuous gradient problems in the convex and nonconvex case, and that includes extrapolation techniques \cite{wen2017linear} in order to accelerate the method.

In this paper, we consider the following minimization problem:
\begin{align}\label{Eq:P}
\tag{P}
\inf \{ \Psi(x) := f(x) + g(x): x \in \mathbb{R}^d \}.
\end{align}
where $f$ is a nonconvex continuously differentiable function and $g$ is a proper lower-semi-continuous (l.s.c.) convex function. We assume that the optimal value of (\ref{Eq:P}) is finite, that is,  $\Psi^{\ast} := \inf \{ \Psi(u): u \in  \mathbb{R}^d \} > -\infty$.
Problem (\ref{Eq:P}) arises in many applications including compressed sensing \cite{donoho2006compressed}, signal recovery \cite{beck2013sparsity}, phase retrieve problem \cite{luke2017phase}.
One classical algorithm for solving this problem is the Proximal Gradient (PG) method \cite{parikh2014proximal}:
\begin{align*}
&x^{k+1} = \arg\min_{x} \bigg\{ g(x) + \langle \nabla f(x^k), x-x^k \rangle + \frac{1}{2\lambda_k}\|x-x^k\|^2 \bigg\}, k \in \mathbb{N},
\end{align*}
where $\lambda_k$ is the stepsize on each iteration.
Proximal gradient method and its variants \cite{schmidt2011convergence,jiang2012inexact, xiao2014proximal,nitanda2014stochastic, chen2012fast,vanli2018global} have been one hot topic in optimization field for a long time due to their simple forms and lower computation complexity.

One branch of developing new PG methods was devoted to accelerations. Accelerated proximal algorithms \cite{beck2009fast,toh2010accelerated} on convex problems have shown to be quite efficient. They were also useful for solving nonconvex problems \cite{wen2017linear,ghadimi2016accelerated,li2015accelerated,carmon2018accelerated}.
For solving nonconvex problems (\ref{Eq:P}), one simple and efficient strategy is to perform extrapolation for each $k \in \mathbb{N}$, with the following form(where $x^{-1}=x^0$)
\begin{equation*}
\left \{
\begin{aligned}
&y^k = x^k + \beta_k(x^k - x^{k-1}),\\
&x^{k+1} = \arg\min_{x} \bigg\{ g(x) + \langle \nabla f(y^k), x-y^k \rangle + \frac{1}{2\lambda_k}\|x-y^k\|^2 \bigg\},
\end{aligned}
\right.
\end{equation*}
where $\lambda_k$ is the stepsize on each iteration, and $\beta_k(x^k - x^{k-1})$ is an extrapolation term.
The previous iteration is called the Proximal Gradient algorithm with Extrapolation (PGe), which have been shown in \cite{wen2017linear} that converges and performs quite well by setting parameters $\beta_k$ properly.
However, PGe has one restriction on solving problem (\ref{Eq:P}): it requires the continuously differentiable part $f$ to be globally Lipschitz gradient continuous on $\mathbb{R}^d$. In fact, this requirement cannot often be satisfied for many practical problems, such as quadratic inverse problem in phase retrieve \cite{luke2017phase} and Poisson linear inverse problems \cite{bertero2009image}, that arise in many real world applications.

In this paper, we propose a new algorithm ---\emph{Bregman Proximal Gradient algorithm with Extrapolation (BPGe)}--- to solve problem (\ref{Eq:P}) without requiring globally Lipschitz gradient continuity of $f$ for each $k \in \mathbb{N}$, from $x^{-1}=x^0$:
\begin{equation*}
\left \{
\begin{aligned}
&y^k = x^k + \beta_k(x^k - x^{k-1}),\\
&x^{k+1} = \arg\min_{x} \bigg\{ g(x) + \langle \nabla f(y^k), x-y^k \rangle + \frac{1}{\lambda_k}D_h(x,y^k)\bigg\},
\end{aligned}
\right.
\end{equation*}
where $D_h$ is a Bregman distance defined in Section~\ref{Sec:2}.
On the basis of Bregman distance theory, we utilize a smooth adaptive condition introduced in \cite{bolte2018first}, which generalizes Lipschitz gradient continuous condition. This smooth adaptive condition was originally proposed to analyze Bregman Proximal Gradient (BPG) algorithm in \cite{bolte2018first}. It can also be used to analyze the convergence of BPGe, since BPGe algorithm extends BPG one by performing extrapolation. In particular, we have that:
\begin{enumerate}[(i)]
\item When $D_h(x,y) = \frac{1}{2}\|x-y\|^2$ and $\beta_k = 0$, BPGe reduces to PG.
\item When $D_h(x,y) = \frac{1}{2}\|x-y\|^2$, BPGe reduces to PGe;
\item When $\beta_k = 0 ~\text{for any}~ k \geq 0$, BPGe reduces to BPG (no extrapolation).
\end{enumerate}
Therefore, PG, PGe and BPG are particular cases of BPGe algorithm.

Recently, other acceleration algorithms for BPG have been proposed in literature, like using it combined with inertial methods~\cite{boct2016inertial} (which is a different methodology from ours), or combining it with Nesterov's acceleration method \cite{hanzely2018accelerated} but requiring the Bregman distance function satisfying some extra crucial triangle scaling property.

From the convergence analysis (Section~\ref{Sec:4}), the BPGe algorithm has to satisfy the condition $D_h(x^k,y^k) \leq \rho \,  C_k \, D_h(x^{k-1},x^k)$ (where $C_k \in (0,1]$ and $\rho \in (0,1)$ are two parameters) to guarantee the convergence.
In the Lipschitz gradient continuous condition $D_h(x,y) = \frac{1}{2}\|x-y\|^2$, this condition is easily satisfied just by choosing $\inf_{k\in \mathbb{N}}\{\beta_k\} \leq \sqrt{\rho \, C}$. But when $D_h$ is general, computing a threshold of $\inf_{k\in \mathbb{N}}\{\beta_k\}$ directly may be hard and expensive. Therefore, we modify this idea to achieve this condition through a line search method (Algorithm~2 introduced in Section~\ref{Sec:3}).

In the convergence analysis, we prove that any limit point of the sequence generated by BPGe is a stationary point under very general conditions. Moreover, by adding some slightly stronger assumptions and Kurdyka-{\L}ojasiewicz property, we could guarantee the sequence generated by BPGe converges to a stationary point.

The paper is organized as follows. We first introduce in Section~\ref{Sec:2}
some basic definitions in optimization, smooth adaptive condition, relative weak convexity, and Kurdyka-{\L}ojasiewicz property.
In Section~\ref{Sec:3}  we introduce the new BPGe algorithm.
The convergence analysis is done in Section~\ref{Sec:4}, where under some assumptions of the smooth adaptive condition and relative weak convexity of problem (\ref{Eq:P}), we first show a descent-type lemma, from which the fact that any limit point of the sequence generated by BPGe is a critical point follows. Later, we prove that the whole sequence generated by BPGe converges to a critical point under Kurdyka-{\L}ojasiewicz property and a stronger assumption.
Several numerical experiments are shown in Section~\ref{Sec:5} to show the performance of the BPGe method compared with the BPG one.

\section{Preliminaries}
\label{Sec:2}
Throughout the paper we will use the following basic notations. Let $\mathbb{N}:= \{0, 1, 2, \dots\}$ be the set of nonnegative integers. We will always work in the Euclidean space $\mathbb{R}^d$, and
the standard Euclidean inner product and the induced norm on $\mathbb{R}^d$ are denoted by $\langle \cdot,\cdot\rangle$ and $\| \cdot\|$, respectively. We denote $B_{\rho}(\tilde{x}):=\{x \in \mathbb{R}^d : \|x-\tilde{x}\|\leq \rho \}$ as the ball of radius $\rho>0$ around $\tilde{x} \in \mathbb{R}^d$, $\text{dist}(x, \mathcal{S}):= \inf_{y \in \mathcal{S}}\|x-y\|$ as the distance from a point $x \in \mathbb{R}^d$ to a nonempty set $\mathcal{S} \subset \mathbb{R}^d$. The domain of the function $f : \mathbb{R}^d \rightarrow (-\infty,+\infty]$ is defined by $\text{dom}~f = \{x \in \mathbb{R}^d: f(x)< +\infty\}$. We say that $f$ is proper if $\text{dom}~f \neq \emptyset$.
For other generalized notions and definitions we refer to \cite{bolte2018first,rockafellar2015convex, rockafellar2009variational}.

\subsection{Smooth Adaptable Function and Relative Weakly Convexity}

In this subsection, we define the notion of \emph{smooth adaptable} condition for nonconvex $f$ proposed in \cite{bolte2018first}.
This property was extended from the recent work~\cite{bauschke2016descent} in which the differentiable functions need to be convex.
This condition is similar to the relative smoothness condition introduced in \cite{lu2018relatively}, but the relative smoothness is based on the fact that $f$ is convex. As we want also to deal with nonconvex functions, in our paper we use the smooth adaptable condition to generalize Lipschitz gradient continuity and to derive the related convergence results of BPGe.

We first introduce the concept of Bregman distance needed in the definition of smooth adaptable condition.

\begin{definition}(Kernel Generating Distance and Bregman Distance \cite{bolte2018first}) \label{Definition:Breg}
 Let $S$ be a nonempty, convex and open subset of $\mathbb{R}^d$. Associated with $S$, a function $h:\mathbb{R}^d  \rightarrow (-\infty, \infty] $  is called a \emph{kernel generating distance} if it satisfies the following:
\begin{enumerate}[(i)]
\item $h$ is proper, lower-semi-continuous and convex, with $\rm{dom}~h \subset \overline {S}$ and $\rm{dom}~\partial h = S$.
\item $h$ is $\mathcal{C}^1$ on $\rm{int}~\rm{dom}~h \equiv S$.
\end{enumerate}
 We denote the class of kernel generating distances by $\mathcal{G}(S)$.
Given $h \in \mathcal{G}(S)$, the \emph{Bregman distance} \cite{bregman1967relaxation} is defined by $D_h:\rm{dom}~h \times \rm{int}~\rm{dom}~h \rightarrow [0, +\infty)$
$$D_h(x, y) := h(x) - h(y)- \langle \nabla h(y),x-y \rangle. $$
\end{definition}
Note that the Bregman distance is, obviously, a proximity measure that measures the proximity of $x$ and $y$.
Next, we list some basic properties of the Bregman distance \cite{chen1993convergence, teboulle2018simplified}:
\begin{enumerate}[(i)]
\item For any $(x, y) \in \text{dom}~h \times \text{int}~\text{dom}~h$, $D_h(x, y) \geq 0.$ If in addition $h$ is strictly convex, $D_h(x, y) = 0$ if and only if $x = y$ holds.
\item \textbf{The three point identity}: For any $y, z \in   \text{int}~\text{dom}~h$ and $ x \in \text{dom}~h$,
$$ D_h(x, z)- D_h(x, y) - D_h(y, z) = \langle \nabla h(y)- \nabla h(z), x-y \rangle.$$
\item \textbf{Linear Additivity}: For any $\alpha, \beta \in \mathbb{R}$, and any functions $h_1$ and $h_2$ we have:
$$ D_{\alpha h_1 + \beta h_2}(x, y) = \alpha D_{h_1}(x, y) + \beta D_{h_2}(x, y), $$
for all couple $(x, y) \in ( \text{dom}~ h_1 \cap \text{dom} ~h_2)^2$ such that both $h_1$ and $h_2$ are differentiable at $y$.
\end{enumerate}

Throughout the paper we will focus on the pair of functions $(f, h)$ that satisfies the smooth adaptable condition. Next we present the definition introduced in \cite{bolte2018first}).
\begin{definition}\label{Definition:Smad}(L-smooth adaptable \cite{bolte2018first})
A pair of functions $(f, h)$, such that $h \in \mathcal{G}(S)$, $f: \mathbb{R}^d \rightarrow (-\infty,+\infty]$ is  a proper and lower-semi-continuous function with $\rm{dom}~h \subset \rm{dom}~f$, which is continuously differentiable on $S= \rm{int}~\rm{dom}~h$,
 is called \emph{$L$-smooth adaptable} ($L$-smad) on $S$ if there exists $L > 0$ such that $Lh-g$ and $Lh+g$ are convex on $S$.
\end{definition}
According to \cite[Lemma 2.1]{bolte2018first}, the pair of functions $(f, h)$ is $L$-smad on $S$ if and only if $\|f(x)-f(y)- \langle \nabla f(y), x-y \rangle \| \leq L \, D_h(x, y)$ for any $(x, y) \in \text{int}~\text{dom}~h$.
When $h(x)= \frac{1}{2}\|x\|^2$ and consequently $D_h(x,y)=\frac{1}{2}\|x-y\|^2$, the $L$-smad condition of $f$ would be reduced to Lipschitz gradient continuity: $\|f(x)-f(y)- \langle \nabla f(y), x-y \rangle \| \leq \frac{L}{2}\|x-y\|^2$ for any $(x, y) \in \text{dom}~h$.

Next we introduce the definition of a $\mu$-relative weakly convex function, given in \cite{davis2018stochastic}. This definition extends the definition of weakly convexity \cite{nurminskii1973quasigradient}, which was employed in the analysis of nonconvex optimization methods.

\begin{definition}\label{Definition:WC}
$f$ is called \emph{$\mu$-relative weakly convex} to $h$ on $S$ if there exists $\mu > 0$ such that $ f + \mu h$ is convex on $S$.
\end{definition}
When $f$ is convex, $\mu = 0$. When $(f, h)$ is $L$-smad on $S$, obviously $f$ is $L$-relative weakly convex to $h$. So by default, $\mu \leq L$.

\subsection{Kurdyka--{\L}ojasiewicz Property}
Finally, we introduce the definition of the Kurdyka--{\L}ojasiewicz property proposed in \cite{bolte2014proximal}. We need this property to prove the global convergence of the whole sequences generated by BPGe for solving (\ref{Eq:P}).
\begin{definition}\label{Definition:KL}(Kurdyka--{\L}ojasiewicz property\cite{bolte2014proximal})
Let $f : \mathbb{R}^d \rightarrow (-\infty,+\infty]$ be a proper lower-semi-continuous function.
\begin{enumerate}[(i)]
\item The function $f$ is said to have the \textit{Kurdyka--{\L}ojasiewicz(KL) property} at $\bar{x} \in \rm{dom}~ \partial f := \{ x \in \mathbb{R}^d: \partial f (x) \neq \emptyset \} $ if there exist $\eta \in (0, +\infty]$, a neighborhood $U$ of $\bar{x}$ and a function $\psi:(0, \eta) \rightarrow \mathbb{R}_{+}$
satisfying:
$$\psi(0) = 0, \psi \in \mathcal{C}^1(0,\eta) \text{~and continuous at ~}0, ~\text{for~all}~ s \in (0,\eta): ~ \psi^{\prime}(s) > 0$$
such that for all $ x \in U \cap [f(\bar{x}) < f(x) < f(\bar{x}) + \eta]$, the following inequality holds
$$ \psi^{\prime}(f(x)-f(\bar{x})) \cdot {\rm{dist}}(0,\partial f(x)) \geq 1. $$
\item If $f$ satisfies the KL property at each point of $\rm{dom}~ \partial f$ then $f$ is called a \textit{KL function}.
\end{enumerate}
\end{definition}

\section{Bregman Proximal Gradient Algorithm with Extrapolation (BPGe)}
\label{Sec:3}
Throughout this paper, we focus on the nonconvex problem (\ref{Eq:P}) in Section \ref{Sec:1} with the following assumptions on $f$ and on the kernel generating distance function $h$:
$h \in \mathcal{G}(\mathbb{R}^d)$, $(f, h) ~\text{is} ~L\text{-smad}$ and $f$ is $\mu\text{-weakly convex relative to}~ h$ (see Definition \ref{Definition:Smad} and \ref{Definition:WC}).
And we also make the following general Assumptions~\ref{Assumption:B} and \ref{Assumption:A} as default.

Assumption~\ref{Assumption:B} is a quite standard condition \cite{bolte2018first} to guarantee the existence of the solution to each step of the optimal subproblem of Proximal Gradient (PG) algorithms.
\begin{assumption}\label{Assumption:B}
The function $\Psi$ is supercoercive, that is,
$$ \lim_{\|u\|\rightarrow \infty} \frac{\Psi(u)}{\|u\|} = \infty.$$
\end{assumption}
Assumptions \ref{Assumption:A} is a general assumption used in the analysis of Bregman Proximal-type algorithms  \cite{bauschke2016descent,chen1993convergence}.
\begin{assumption}\label{Assumption:A}
\begin{enumerate}[(i)]
\item $h$ is strictly convex.\label{It1:1}
\item If $\{x^k\}_{k\in \mathbb{N}}$ converges to some $x$ in $\rm{dom}~h$ then $D_h(x,x^k)\rightarrow 0$. \label{It1:2}
\item If $\{x^k\}_{k\in \mathbb{N}},\{y^k\}_{k\in \mathbb{N}}$ defined in $\rm{dom}~ h$ are sequences such that $ y^k \rightarrow x^{\ast} \in \overline{\rm{dom}~ h}$, $\{x^k\}_{k\in \mathbb{N}} $ is bounded, and if $D_h(x^k,y^k) \rightarrow 0$, then $ x^k \rightarrow x^{\ast}$.\label{It1:3}
\end{enumerate}
\end{assumption}

We are now ready to introduce our BPGe algorithm, divided in two parts, Algorithm~1 and Algorithm~2.
Algorithm~1 is the whole framework for solving Problem (\ref{Eq:P}). And Algorithm~2 is a line search step, which is used to search a proper parameter $\beta_k$ at every iteration in Algorithm~1.
Throughout the whole paper, we make the following notations
$$\overline{\lambda} := \sup_{k\in \mathbb{N}} \{ \lambda_k \}, ~~~~\underline{\lambda} := \inf_{k\in \mathbb{N}} \{ \lambda_k \}.$$
By default $0< \underline{\lambda}  \leq \overline{\lambda} < \infty$.

\begin{tcolorbox}
{\textbf{Algorithm 1: BPGe}---Bregman Proximal Gradient algorithm with Extrapolation.}
\hrule
\begin{algorithmic}
\STATE{\textbf{Data:} A function $h$ defined in Definition \ref{Definition:Breg}}\\
\qquad such that $(f, h)$ is $L$-smad holds\\
\qquad and $f$ is $\mu$-weakly convex relative to $h$ on $\mathbb{R}^d$.\\
\quad Error tolerance: \texttt{TOL}.
\STATE{\textbf{Initialization:} $x^0 = x^{-1} \in \text{int}~\text{dom}~h$ and $0< \lambda_k \leq 1/L$.}
\STATE{\textbf{General step:}}\\
\quad \texttt{For} $k =0, 1, 2, \dots$, $k_{max}$ \texttt{repeat} \\
\quad \quad  Take
\begin{equation}
y^{k}  =  x^{k}+\beta_{k}(x^{k}-x^{k-1}),
\end{equation}
\quad \quad where $\beta_{k}$ is searched according to \textbf{Line Search} in Algorithm~2. \\
\quad \quad Then compute
\begin{align}
\hspace*{-0.5cm}
x^{k+1} \in  \arg\min \left\{x: g(x)+ \left\langle x-y^k,\nabla f(y^k)\right\rangle + \frac{1}{\lambda_k}D_h(x, y^k), x \in \mathbb{R}^d \right\}. \label{Eq:Alg2_2}
\end{align}
\quad {\texttt{until}} \texttt{EXIT(TOL)} received.
\end{algorithmic}
\end{tcolorbox}

\begin{tcolorbox}
{\textbf{Algorithm 2: Line Search} for Algorithm~1 at the $k$-th iteration.}\\
\hrule
\begin{algorithmic}
\STATE{\textbf{Data:} A function $h$ defined in Algorithm~1, fix $0 < \eta < 1$, ${\beta_0} \in [0,1)$, $0< \rho < 1$.}
\STATE{\textbf{Input:} $x^{k-1}, x^{k} \in \text{int}~\text{dom}~h$, $C_{k} = \frac{\lambda_{k}^{-1}}{\lambda_{k}^{-1}+\mu}$.}
\STATE{\textbf{General step:}}\\
\quad$\tilde{\beta} = \beta_0$ \\
\quad \texttt{While} $D_h(x^{k},x^{k} + \tilde{\beta}(x^{k}-x^{k-1})) > \rho \, C_{k} \, D_h(x^{k-1}, x^{k})$ \texttt{do}
     $$ \tilde{\beta} = \eta\tilde{\beta}. $$
\STATE{ \textbf{Return:} Set the feasible step size $\beta_{k} = \tilde{\beta}$ for iteration $k$.}
\end{algorithmic}
\end{tcolorbox}

We remark that an important point on any iterative process is to define suitable error control techniques. In this paper we consider a quite simple strategy in order to determine the  \texttt{EXIT} conditions. On one hand we fix a maximum number of iterations $k_{max}$ (in most of our tests $5000$ iterations) and  \texttt{EXIT(TOL)=true} if $ \|x^k-x^{k-1}\|/\max\{1, \|x^k\| \} \leq \texttt{TOL}$ (in our tests  $\texttt{TOL} = 10^{-6}$ as in \cite{wen2017linear}). Other option is to check the convergence using the objective function, instead of the solution itself, that is $ \|\Psi(x^k)-\Psi(x^{k-1)}\|/\max\{1, \|\Psi(x^k)\| \} \leq \texttt{TOL}$.

We first verify that (\ref{Eq:Alg2_2}) is well-defined using the following  Proposition~\ref{Proposition:1}. For all $y \in \text{int}~\text{dom}~h$ and stepsize $0< \lambda \leq 1/L$, we define the Bregman proximal gradient mapping as:
\begin{align}
T_{\lambda}(y):= \arg\min \big\{ g(u) + \langle \nabla f(y), u-y \rangle + \lambda^{-1}D_h(u,y): u \in \mathbb{R}^d \big\}.
\end{align}
In Proposition \ref{Proposition:1} we prove that $T_{\lambda}(y)$ is well posed. Thus by Proposition \ref{Proposition:1}, $x^{k+1} \in T_{\lambda_k}(x^{k})$, and fixing $ \inf \{\lambda_k\} > 0$, then Step (\ref{Eq:Alg2_2}) in BPGe algorithm is well-defined.

\begin{proposition}\label{Proposition:1}
Suppose that Assumption~\ref{Assumption:B} holds, let $y \in \rm{int}~\rm{dom}~h$ and $0< \lambda \leq 1/L$. Then, the set $T_{\lambda}(y)$ is a nonempty and compact set.
\end{proposition}
\begin{proof}
Fix any $y \in \text{int}~\text{dom}~h$ and $0< \lambda \leq 1/L$. For any $u \in \mathbb{R}^d$, we define
$$\Psi_h(u) = g(u)+ f(y) + \big\langle u-y ,\nabla f(y)\big\rangle +\lambda^{-1} D_h(u,y), $$
so that $T_{\lambda}(y) = \arg\min_{u\in \mathbb{R}^d} \Psi_h(u),$
It can also be represented as
\begin{align*}
\Psi_h(u) &= \Psi(u) -f(u)+f(y)+ \big\langle u-y ,\nabla f(y)\big\rangle +\lambda^{-1} D_h(u,y) \\
& \geq  \Psi(u) + L \, D_h(u,y)-\bigg[f(u)-f(y)- \big\langle u-y ,\nabla f(y)\big\rangle \bigg]\\
& \geq \Psi(u).
\end{align*}
where the second inequality is obtained by taking into account $\lambda^{-1} \geq L$ and in the last inequality that $(f,h)$ is $L$-smooth adaptable.
According to Assumption~\ref{Assumption:B}, i.e. $\lim_{\|u\| \rightarrow \infty} \Psi(u) = \infty$, there is
$$ \lim_{\|u\| \rightarrow \infty}\Psi_h(u) \geq \lim_{\|u\| \rightarrow \infty}\Psi(u) =\infty.$$
 Since $\Psi_h$ is also proper and lower-semi-continuous, invoking the modern form of Weierstrass’ theorem (see, e.g., \cite[Theorem 1.9, page 11]{rockafellar2015convex}), it follows that the value $\inf_{\mathbb{R}^d} \Psi_h$ is finite, and the set $\arg\min_{u \in \mathbb{R}^d} \Psi_h(u) \equiv T_{\lambda}(y)$ is nonempty and compact.
\end{proof}

Secondly, we add an extrapolation step to the BPGe algorithm to choose a suitable $\beta_k$ at each iteration step through the line search Algorithm~2. On this step it is hard to guarantee directly the decrease of function value $\Psi(x^k)$. Therefore, we focus on guaranteeing sufficient decrease of the Lyapunov sequences defined in Section~\ref{Sec:4} in the convergence analysis. However, it still requires an extra  condition $D_h(x^{k}, x^{k} + \beta_{k}(x^{k}-x^{k-1})) \leq \rho \, C_k \,  D_h(x^{k-1}, x^{k})$.
When $h=\frac{1}{2}\|x\|^2$, BPGe is reduced to the PGe algorithm \cite{wen2017linear} and this condition is easily satisfied by setting $0 \leq \beta_k \leq \sqrt{\rho \frac{L}{L+\mu}}$. But when $h$ is more general and complex, computing the threshold of $\beta_k$ directly may be hard and expensive. So,  we try to reach this condition by a line search method introduced in Algorithm~2. Thus, our next step is to  verify that Algorithm~2 is well-defined, as the following proposition \ref{Lemma:FT} shows.

\begin{proposition}\label{Lemma:FT}(Finite termination of Algorithm~2).
Consider Algorithm~1 and fix $k \in \mathbb{N}$. Let $0 < \eta < 1, 0< \rho < 1$, $\tilde{\beta} \in [0,1)$, $C_k = \frac{\lambda_{k}^{-1}}{\lambda_{k}^{-1}+\mu} > 0$. Then, there exists $J \in \mathbb{N}$ such that $\beta_{k} := \eta^j\tilde{\beta}$  satisfies $$D_h(x^{k},x^{k} + {\beta_k}(x^{k}-x^{k-1})) \leq \rho \, C_k \,  D_h(x^{k-1}, x^{k})$$ for any $j \geq J$.
\end{proposition}
\begin{proof}
This result is proved by contradiction.
Suppose that $$D_h(x^{k},x^{k} + \eta^j\tilde{\beta}(x^{k}-x^{k-1})) > \rho \,  C_k \, D_h(x^{k-1}, x^{k})$$ holds for any $j \in \mathbb{N}$.

When $x^k = x^{k-1}$, Algorithm~2 terminates with $\beta_k = \tilde{\beta}$ directly.

When $x^k \neq x^{k-1}$, since $$\| x^{k}-(x^{k} + \tilde{\beta}(x^{k}-x^{k-1}))\| =  \eta^j\tilde{\beta}\|x^{k}-x^{k-1}\|  \rightarrow 0, \quad  j\rightarrow \infty,$$ according to Assumption~\ref{Assumption:A}(\ref{It1:2}), $D_h\big(x^{k},x^{k} + \eta^j\tilde{\beta}(x^{k}-x^{k-1})\big) \rightarrow 0$.
Thus for any $\varepsilon >0$, there exist a number $J \in \mathbb{N}$ such that
$$ D_h(x^{k},x^{k} + \eta^j\tilde{\beta}(x^{k}-x^{k-1})) < \varepsilon, ~\text{for all}~ j \geq J.$$
Since $x^{k} \neq x^{k-1}$, and due to the strictly convexity of $h$ in Assumption~\ref{Assumption:A}(\ref{It1:1}), $$D_h(x^{k-1},x^{k}) > 0.$$
If we set $ \varepsilon = \frac{1}{2}\rho \, C_k \, D_h(x^{k-1},x^{k})$, then
$$ \rho \, C_k \, D_h(x^{k-1},x^{k}) < D_h\bigg(x^{k},x^{k} + \eta^j\tilde{\beta}(x^{k}-x^{k-1})\bigg) < \frac{1}{2} \, \rho \, C_k \, D_h(x^{k-1}, x^{k}),$$
for $j \geq J$, which is a contradiction.
\end{proof}

\section{Convergence Analysis of BPGe}
\label{Sec:4}
In this section we provide the main convergence results of the BPGe algorithm.
First of all, following the analysis of Remark 4.1(ii) in \cite{bolte2018first}, we obtain the following Lemma \ref{Lemma:1}.
We find that after adding an extrapolation term, it is hard to justify monotonicity of the objective function $\Psi$ directly. But for a special auxiliary sequence, defined by
$$ H_{k,M} := \Psi(x^{k}) + M D_h(x^{k-1},x^{k}), \quad M>0, \quad \forall k \in \mathbb{N}$$
the monotone property will be presented in our settings.

\begin{lemma} \label{Lemma:1}
For any $x \in \rm{int}~\rm{dom}~h$, and let be a sequence $\{x^k\}_{k\in \mathbb{N}}$ produced by BPGe, then
\begin{enumerate}[(i)]
	\item For any  $ k \in \mathbb{N}$, we have
	\begin{align}\label{Eq:Lemma_1}
	\Psi(x^{k+1}) - \Psi(x) \leq  (\lambda_k^{-1}+ \mu) \, D_h(x,y^{k})- \lambda_k^{-1}D_h(x,x^{k+1}) - (\lambda_k^{-1}- L) \, D_h(x^{k+1},y^k).
	\end{align}
	\item For any  $ k \in \mathbb{N}$, we have
	\begin{align}
	H_{k+1,M} - H_{k,M} \leq (M - \lambda_k^{-1})\, D_h(x^k,x^{k+1})- \left(M - \rho \lambda_k^{-1} \right) \, D_h(x^{k-1},x^{k}). \label{Eq:Lem2_1}
	\end{align}
     Moreover, assuming there exists some $M$ such that $\rho \, \underline{\lambda}^{-1} \leq M \leq \overline{\lambda}^{-1}$, then the sequence $\{H_{k,M}\}$ is nonincreasing and convergent for the fixed $M$.
\end{enumerate}
\end{lemma}
\begin{proof}
\noindent(i) According to the first order condition of (\ref{Eq:Alg2_2}), we get
$$0\in \partial g(x^{k+1})+ \nabla f(y^k) +\lambda_k^{-1}\big(\nabla h(x^{k+1})-\nabla h(y^k)\big), \quad \forall k \in \mathbb{N}.$$
Combining with the convexity of $g$, there is
\begin{align*}
g(x) - g(x^{k+1}) \geq \bigg\langle -\nabla f(y^k) -\lambda_k^{-1}\big(\nabla h(x^{k+1})-\nabla h(y^k)\big), \, x-x^{k+1} \bigg\rangle, \quad \forall k \in \mathbb{N}.
\end{align*}
Together with the three point identity of Bregman distance
$$\lambda_k^{-1}\big\langle \nabla h(x^{k+1})- \nabla h(y^k), \, x-x^{k+1}  \big\rangle =
\lambda_k^{-1} \big(D_h(x,y^{k})-D_h(x,x^{k+1})-D_h(x^{k+1},y^k)\big)$$
we have that
\begin{align}
&g(x) - g(x^{k+1}) +f(x) -f(x^{k+1}) \nonumber \\
& \quad \geq  f(x) -f(x^{k+1}) - \bigg\langle \nabla f(y^k), \, x - x^{k+1} \bigg\rangle \nonumber \\ & \qquad \qquad  - \lambda_k^{-1} \bigg(D_h(x,y^{k})-D_h(x,x^{k+1})-D_h(x^{k+1},y^k)\bigg), \quad \forall k \in \mathbb{N}.
\end{align}
If we take the $\mu$-relative weakly convex property and $L$-smad property of $(f,h)$ (see Definitions~\ref{Definition:Smad} and \ref{Definition:WC}),
\begin{align}
&f(x) -f(x^{k+1}) - \big\langle \nabla f(y^k), \, x - x^{k+1} \big\rangle \nonumber \\
& \quad = f(x)-f(y^k)-\bigg\langle \nabla f(y^k), \, x -y^k \bigg\rangle + f(y^k) -f(x^{k+1}) - \bigg\langle \nabla f(y^k), \, y^k - x^{k+1}  \bigg\rangle \nonumber \\
& \quad \geq -\mu \, D_h(x,y^k) - L \, D_h(x^{k+1},y^k),  \qquad \forall k \in \mathbb{N}.
\end{align}
Thus
\begin{align*}
\Psi(x^{k+1}) - \Psi(x) \leq  (\lambda_k^{-1}+ \mu) \, D_h(x,y^{k})- \lambda_k^{-1}D_h(x,x^{k+1}) - (\lambda_k^{-1}-L) \, D_h(x^{k+1},y^k) .
\end{align*}
\vspace{12pt}

\noindent(ii)  For any $ k \in \mathbb{N}$, taking $x=x^k$ into (\ref{Eq:Lemma_1}), together with $L \leq \lambda_k^{-1}, D_h(x^{k+1},y^k) \geq 0$ we get
\begin{align*}
 \Psi(x^{k+1}) - \Psi(x^k) \leq (\lambda_k^{-1}+ \mu) \, D_h(x^k,y^{k})- \lambda_k^{-1}D_h(x^k,x^{k+1}).
\end{align*}
If $ x^k = x^{k-1}$, we get $y^k = x^k$, thus $ D_h(x^k,y^k) = D_h(x^{k-1},x^k) =0$ and
\begin{align}
\Psi(x^{k+1}) + \lambda_k^{-1}D_h(x^k,x^{k+1}) \leq \Psi(x^k) = \Psi(x^k)+ (\lambda_k^{-1}+ \mu) \, \rho \, C_k \, D_h(x^{k-1}, x^k). \label{Eq:Lem2_3}
\end{align}
If $x^k \neq x^{k-1}$, according to Algorithm~2, we have $ D_h(x^k,y^k) \leq \rho \, C_k \, D_h(x^{k-1},x^k)$, thus
\begin{align}
\Psi(x^{k+1}) + \lambda_k^{-1}D_h(x^k,x^{k+1}) \leq \Psi(x^k)+ (\lambda_k^{-1}+ \mu) \, \rho \, C_k \, D_h(x^{k-1}, x^k).
\end{align}
Combining these two cases, we obtain
$$ \Psi(x^{k+1}) + \lambda_k^{-1}D_h(x^k,x^{k+1}) \leq \Psi(x^k)+ (\lambda_k^{-1}+ \mu) \, \rho \, C_k \, D_h(x^{k-1}, x^k), \, \, \forall k \in \mathbb{N}.$$

From the definition of $H_{k,M}$, we see that
\begin{align*}
H_{k+1,M} - H_{k,M} \leq (M - \lambda_k^{-1})D_h(x^k,x^{k+1})- \left(M - \rho \lambda_k^{-1} \right)D_h(x^{k-1},x^{k}), \, \, \forall k \in \mathbb{N}.
\end{align*}
Furthermore, assuming there exists some $M$ such that $$ \rho \, \lambda_k^{-1} \leq  \rho \, \underline{\lambda}^{-1} \leq M \leq \overline{\lambda}^{-1} \leq \lambda_k^{-1},$$ and fixing one of such values of $M$, we find that
$$ H_{k+1,M} - H_{k,M} \leq 0, \quad \forall k \in \mathbb{N}, $$
that is, $\{H_{k,M}\}_{k\in \mathbb{N}}$ is nonincreasing for the fixed value of $M$.

Recall that $H_{k,M} \geq \inf \Psi > -\infty$ and $H_{k,M}$ is nonincreasing. This implies that $\{H_{k,M}\}$ is convergent for some fixed $M$.
\end{proof}

The next corollary is an obvious result based on Lemma \ref{Lemma:1}. We analyze the boundness of the sequences produced by BPGe algorithm. Since $H_{k,M}$ is nonincreasing according to Lemma~\ref{Lemma:1}(ii), it is easy to verify that the sequence $\{x^k \}_{k\in \mathbb{N}}$ generated by BPGe is bounded according to Assumption~\ref{Assumption:B}. The boundness would act as a tool in the following analysis, so we present this result as the auxiliary Corollary~\ref{Corollary:1}.
\begin{corollary}\label{Corollary:1}
Assume there exists some $M$ such that $\rho \, \underline{\lambda}^{-1} \leq M \leq \overline{\lambda}^{-1}$, then the sequence $\{x^k \}_{k\in \mathbb{N}}$ generated by BPGe is bounded.
\end{corollary}

If the stepsize $\lambda_k$ and parameter $\rho$ in Algorithm~2 satisfy $\rho < \overline{\lambda}^{-1}/\underline{\lambda}^{-1} = \underline{\lambda}/\overline{\lambda}$,
then we could get sufficient decrease of the auxiliary sequence  $\{H_{k,M}\}_{k\in \mathbb{N}}$ for the fixed $M$ given in Lemma \ref{Lemma:1}.
As a consequence, we can bound the sum of Bregman distance between two iteration points generated by BPGe.
Moreover, adding stronger assumptions than Assumption \ref{Assumption:A} on the kernel generating distance $h$, such as strong convexity, we could get that $\lim_{k\rightarrow \infty} \|x^k- x^{k-1}\| =0$ for the sequence  $\{x^k \}_{k\in \mathbb{N}}$ in $\mathbb{R}^d$ by BPGe. In this paper, we just consider the set of weaker blanket Assumptions~\ref{Assumption:B} and \ref{Assumption:A}, that permit us to prove that any limit point of the sequence $\{x^k\}_{k\in \mathbb{N}}$ generated by BPGe, if exists, is a stationary point of the objective function $\Psi$.

Assume that  $\{x^k \}_{k\in \mathbb{N}}$ is generated from a starting point $x^0$.
The set of all limit points of $\{x^k \}_{k\in \mathbb{N}}$ is denoted by
$$ \omega(x^0):= \{ \overline{x}: \text{~an increasing sequence of integers~} \{k_i\}_{i \in \mathbb{N}} \text{~such that~} x^{k_i} \rightarrow \overline{x}  \text{~as ~} i \rightarrow \infty \}.$$
The next technical lemma shows, among other results, that for any $x^0 \in \mathbb{R}^d$, $\omega(x^0) \subseteq \text{crit}~\Psi $ holds.

\begin{lemma}\label{Theorem:1}
Suppose $\rho < \underline{\lambda}/\overline{\lambda}$ and let $\{x^k\}_{k\in \mathbb{N}} $ be a sequence generated from $x^0$ by BPGe. Then the following statements hold:
\begin{enumerate}[(i)]
\item $\sum_{k=0} ^{\infty} D_h(x^{k-1},x^k) < \infty $ and $\lim_{k\rightarrow \infty} D_h(x^{k-1},x^k) = 0$.
\item Any limit point of $\{x^k\}_{k\in \mathbb{N}} $ is a critical point of $\Psi$ ($\omega(x^0) \subseteq \rm{crit}~\Psi $).
\item $\zeta := \lim_{k\rightarrow \infty} \Psi(x^k)$ exists and $\Psi \equiv \zeta$ on $\omega(x^0)$.
\end{enumerate}
\end{lemma}
\begin{proof}
\noindent(i) Since $\rho < \underline{\lambda}/\overline{\lambda}$, we have that $ \rho \, \lambda_k^{-1} \leq \rho \, \underline{\lambda}^{-1} < \overline{\lambda}^{-1}$, and we choose $M \in ( \rho \, \underline{\lambda}^{-1}, \overline{\lambda}^{-1}]$.
From (\ref{Eq:Lem2_1}), together with the nonnegativeness of $D_h(x^k,x^{k+1})$ and $M \leq \lambda_k^{-1} $, we have $\forall k \in \mathbb{N}$
\begin{align}\label{Eq:Lem3_1}
\left(M -  \rho \underline{\lambda}^{-1} \right)D_h( x^{k-1},x^k ) \leq
\left(M -  \rho \lambda_k^{-1} \right)D_h( x^{k-1},x^k )  \leq  H_{k,M}-H_{k+1,M},
\end{align}
which implies, $\forall K \in \mathbb{N}$, that
\begin{align}\label{Eq:Lem3_11}
0 \leq \sum_{i=0}^K  \left(M -  \rho \underline{\lambda}^{-1} \right)D_h(x^{k-1},x^k ) \leq H_{0,M} - H_{K+1,M},
\end{align}
by summing both sides of (\ref{Eq:Lem3_1}) from $0$ to $K$. Since $\{H_{k,M}\}$ is convergent by Lemma~\ref{Lemma:1}(ii), letting $K \rightarrow \infty$, we conclude that the infinite sum exists and is finite, i.e.,
$$\sum_{i=0}^K  \left(M -   \rho \underline{\lambda}^{-1} \right)D_h(x^{k-1},x^k ) < \infty.$$
Since $ M - \rho \underline{\lambda}^{-1} >0$, we obtain directly that $ \sum_{i=0}^K D_h(x^{k-1}, x^k) \leq \infty$ and $ \lim_{k\rightarrow \infty } D_h(x^{k-1}, x^k) =0 $.

\vspace{12pt}

\noindent(ii) Let $\overline{x}$ be a limit point of $\{x^k\}_{k\in \mathbb{N}}$. Let $\{x^{k_i}\}$  be a subsequence such that $\lim_{i\rightarrow \infty} x^{k_i} = \overline{x}$. Since $D_h(x^{k_i-1},x^{k_i}) \rightarrow 0$, and we know $\{x^{k_i-1}\}_{i\in \mathbb{N}}$ is bounded according to Corollary~\ref{Corollary:1}, Assumption~\ref{Assumption:B}(ii) implies $x^{k_i-1} \rightarrow \overline{x}$. Similarly, we get $x^{k_i-2} \rightarrow \overline{x}$. By the representation of $y^{k_i-1} = x^{k_i-1} + \beta_{k_i-1}(x^{k_i-1} - x^{k_i-2})$ or  $y^{k_i-1} = x^{k_i-1}$ (if $ x^{k_i-1} = x^{k_i-2}$ ), we obtain
\begin{align}\label{Eq:Lem3_2}
\|y^{k_i-1} - x^{k_i}\| &\leq \|x^{k_i-1}-x^{k_i}\| + \|x^{k_i-1} - x^{k_i-2}\| \nonumber \\
& \leq \| x^{k_i} -\overline{x} \| +  2\| x^{k_i-1} -\overline{x} \|+  \| x^{k_i-2} -\overline{x} \| \rightarrow 0.
\end{align}
On one hand, we prove that there exists $v^{k_i} \in  \partial \Psi(x^{k_i})$ such that $  v^{k_i} \rightarrow 0$. By using the first-order optimality condition of the minimization problem (\ref{Eq:Alg2_2}), we obtain
\begin{align*}
0 \in \lambda_{k_i-1} \partial g(x^{k_i}) + \lambda_{k_i-1} \nabla f(y^{k_i-1}) + \nabla h(x^{k_i}) - \nabla h(y^{k_i-1}), \quad \forall k_i \in \mathbb{N}.
\end{align*}
Therefore, we observe that
\begin{align}
 \nabla f(x^{k_i}) -\nabla f(y^{k_i-1}) -\lambda_{k_i-1}^{-1}\big(\nabla h(x^{k_i}) - \nabla h(y^{k_i-1})\big) \in \partial \Psi(x^{k_i}), \quad \forall k_i \in \mathbb{N}.  \label{Theorem1_Eq1}
\end{align}
Taking limits on the left hand in (\ref{Theorem1_Eq1}) we have that
\begin{align}
&\|\nabla f(x^{k_i}) -\nabla f(y^{k_i-1}) -\lambda_{k_i-1}^{-1}(\nabla h(x^{k_i}) - \nabla h(y^{k_i-1}))\| \nonumber\\
& \quad \leq \|\nabla f(x^{k_i}) -\nabla f(y^{k_i-1})\|+\underline{\lambda}^{-1}\|\nabla h(x^{k_i}) - \nabla h(y^{k_i-1})\| \rightarrow 0, \quad k_i \rightarrow \infty, \label{Theorem1_Eq2}
\end{align}
where the limit can be got according to (\ref{Eq:Lem3_2}) and the continuity of $\nabla f$ and $\nabla h$. Thus, we get that
there exist $v^{k_i} \in  \partial \Psi(x^{k_i})$ such that $\| v^{k_i}\| \rightarrow 0$ as $k_i \rightarrow \infty$.

On the other hand, we derive that $\Psi(x^{k_i}) \rightarrow \Psi(\overline{x})$, $k_i \rightarrow \infty$. From the lower-semi-continuity of $\Psi$, we have
\begin{align}
 \Psi(\overline{x})  \leq \lim \inf_{i \rightarrow \infty} \Psi(x^{k_i}). \label{Theorem1_Eq3}
\end{align}
According to the iteration step (\ref{Eq:Alg2_2}) of BPGe, for $k_i \geq 1$, we have
\begin{align*}
\lambda_{k_i-1} \, g(x^{k_i})+ \left\langle x^{k_i}-\overline{x}, \lambda_{k_i-1} \nabla f(y^{k_i-1}) \right\rangle + D_h(x^{k_i},y^{k_i-1}) \nonumber\\ \qquad \leq \lambda_{k_i-1} g(\overline{x})+D_h(\overline{x},y^{k_i-1}) .
\end{align*}
Adding $\lambda_{k_i-1}f(x^{k_i})$ to both sides,  we have
\begin{align}\label{Theorem1_Eq4}
&\lambda_{k_i-1}\Psi(x^{k_i})+ \left\langle x^{k_i}-\overline{x}, \lambda_{k_i-1} \nabla f(y^{k_i-1}) \right\rangle + D_h(x^{k_i},y^{k_i-1}) \nonumber\\
& \qquad \leq \lambda_{k_i-1} g(\overline{x})+\lambda_{k_i-1}f(x^{k_i})+D_h(\overline{x},y^{k_i-1}), \quad \forall k_i \in \mathbb{N}.
\end{align}
After rearranging terms, for all $ k_i \in \mathbb{N}$, it follows
\begin{align}\label{Theorem1_Eq5}
\Psi(x^{k_i}) \leq \Psi(\overline{x})+ f(x^{k_i})-f(\overline{x}) - \left\langle x^{k_i}-\overline{x}, \nabla f(y^{k_i-1}) \right\rangle \nonumber\\
\qquad \quad  -\lambda_{k_i-1}^{-1}D_h(x^{k_i},y^{k_i-1}) +\lambda_{k_i-1}^{-1}D_h(\overline{x},y^{k_i-1}).
\end{align}
 $L$-smad property and $\mu$-relative weakly convexity of $(f,h)$ imply that  for all $ k_i \in \mathbb{N}$
 \begin{align}\label{Theorem1_Eq6}
&f(x^{k_i})-f(\overline{x}) - \left\langle x^{k_i}-\overline{x}, \nabla f(y^{k_i-1}) \right\rangle \nonumber\\
&\qquad \quad \leq L \, D_h(x^{k_i},\overline{x}) +  \left\langle x^{k_i}-\overline{x}, \, \nabla f(\overline{x})-\nabla f(y^{k_i-1}) \right\rangle \nonumber\\
&\qquad \quad = L\, D_h(x^{k_i},\overline{x}) +  D_f(x^{k_i},y^{k_i-1}) - D_f(x^{k_i},\overline{x})-D_f(\overline{x},y^{k_i-1}). \nonumber\\
&\qquad \quad \leq L\, D_h(x^{k_i},\overline{x}) + L \, D_h(x^{k_i},y^{k_i-1}) + \mu \, D_h(x^{k_i},\overline{x})+ \mu \,  D_h(\overline{x},y^{k_i-1})
\end{align}
Plugging (\ref{Theorem1_Eq6}) in (\ref{Theorem1_Eq5}), passing to the limit, together with the relationship $\underline{\lambda} \leq \lambda_{k_i} \leq \overline{\lambda}$,  we have
\begin{align}\label{Theorem1_Eq8}
\lim_{i\rightarrow \infty} \Psi(x^{k_i}) &\leq  \Psi(\overline{x})+ \lim_{i\rightarrow \infty}\bigg[(-\overline{\lambda}^{-1}+L)D_h(x^{k_i},y^{k_i-1}) + \nonumber\\ & \qquad \quad (\underline{\lambda}^{-1}+\mu)D_h(\overline{x},y^{k_i-1})+ (L+\mu) D_h(x^{k_i}, \overline{x})\bigg] \nonumber\\
 &\leq  \Psi(\overline{x})+ \lim_{i\rightarrow \infty} (\underline{\lambda}^{-1}+\mu)\left[D_h(\overline{x},y^{k_i-1})+D_h(x^{k_i}, \overline{x})\right],
\end{align}
where the second inequality is based on $L \leq \overline{\lambda}^{-1} \leq \underline{\lambda}^{-1}$ in BPGe. From (\ref{Eq:Lem3_2}), together with the continuity of $\nabla h$, we obtain
\begin{align*}
& \lim_{i\rightarrow \infty} \left[D_h(\overline{x},y^{k_i-1})+D_h(x^{k_i}, \overline{x})\right]  \nonumber\\ &\qquad \quad
\leq \lim_{i\rightarrow \infty} \left[  D_h(\overline{x},y^{k_i-1})+ D_h(y^{k_i-1},\overline{x})+D_h(x^{k_i}, \overline{x}) + D_h(\overline{x},x^{k_i})   \right] \nonumber\\ &\qquad \quad
\leq  \lim_{i\rightarrow \infty} \bigg[\|\nabla h(y^{k_i-1})-\nabla h(\overline{x})\|\| y^{k_i-1}-\overline{x}\| + \|\nabla h(x^{k_i})-\nabla h(\overline{x})\|\| x^{k_i}-\overline{x}\| \bigg] \nonumber\\
&\qquad \quad  = 0.
\end{align*}
Hence we have
\begin{align}\label{Theorem1_Eq9}
\lim \sup_{i\rightarrow \infty} \Psi(x^{k_i}) \leq \Psi(\overline{x}).
\end{align}
Combining (\ref{Theorem1_Eq3}) and (\ref{Theorem1_Eq9}) yields  $\Psi(x^{k_i}) \rightarrow \Psi(\overline{x})$, $k_i \rightarrow \infty$.

Thus, according to these results,  and the closedness of $\partial \Psi$ (see, Exercise 8 in Page 80 \cite{borwein2010convex}), we have $0 \in \partial \Psi(\overline{x})$.
\vspace{12pt}

\noindent(iii)  In view of Lemma \ref{Lemma:1} and (i), the sequence $\{H_{k,M} \}$ is convergent and $D_h(x^{k-1},x^k) \rightarrow 0$, these together with the definition of $H_{k,M}$ imply $\lim_{k\rightarrow \infty} \Psi(x^k)$ exists, denoted as $\zeta$.
According to the last part of the proof in (ii), and taking $\overline{x} \in \omega(x^0)$ with a convergent subsequence $\{x^{k_i}\}$ such  that $\lim_{i\rightarrow \infty} x^{k_i} = \overline{x}$, we know that
 $$\zeta = \lim_{i\rightarrow \infty} \Psi(x^{k_i}) = \Psi(\overline{x}).$$
Thus the conclusion is completed since $\overline{x}$ is arbitrary.
\end{proof}

Next, we prove a global $\mathcal{O}(\frac{1}{K})$ sublinear convergence rate for the sequence $\min_{k\in \mathbb{N}}D_h(x^{k-1},x^k)$ of the algorithm. In fact, the linear convergence rate can also be got if we add more assumptions, like KL property and concrete KL exponent (we refer to \cite{li2017calculus}), based on similar deductions as in  \cite[Theorem 6.3]{bolte2018first}.

\begin{corollary}\label{Corollary:2}
Suppose $\rho < \underline{\lambda}/\overline{\lambda}$ and $\{x^k\}_{k\in \mathbb{N}} $ be a sequence generated from $x^0$ by BPGe. Then for all $K \geq 1$, $\min_{1 \leq k \leq K}D_h(x^{k-1},x^k)$ converges with a sublinear rate as $\mathcal{O}(\frac{1}{K})$.
\end{corollary}
\begin{proof}
Set $M = \overline{\lambda}^{-1}$, recall (\ref{Eq:Lem3_11}), now for $K \geq 1$,
\begin{align*}
0 \leq \sum_{i=1}^K  \left(\overline{\lambda}^{-1} - \rho \, \underline{\lambda}^{-1} \right)D_h(x^{k-1},x^k ) \leq H_{1,M} - H_{K+1,M}.
\end{align*}
Hence we obtain
\begin{align}
 \min_{1 \leq k \leq K} D_h(x^{k-1}, x^k) \leq \frac{H_{1,M} - H_{K+1,M}}{K\left(\overline{\lambda}^{-1} - \rho \, \underline{\lambda}^{-1} \right)}.
\end{align}
\end{proof}

Next, we focus on performing a global convergence analysis. We aim to prove that the sequence $\{x^k\}_{k\in \mathbb{N}}$ generated by BPGe converges to a critical point of the objective function $\Psi$ defined in (\ref{Eq:P}). In order to prove global convergence, we use the proof methodology introduced in reference \cite{attouch2013convergence}. This proof methodology proves global convergence result for several types of nonconvex nonsmooth problems. Other similar forms were referred in many works \cite[Section 3.2]{ochs2014ipiano}, \cite[Section 4]{pock2016inertial}, \cite[Section 4.2]{bolte2018first}.

For the reader’s convenience, we firstly describe the proof methodology summarized in \cite[Theorem 3.7]{ochs2014ipiano} with a few modifications and then we apply it to prove the convergence of BPGe in  Theorem \ref{Theorem:3}.
\begin{theorem} \cite[Theorem 3.7]{ochs2014ipiano} \label{Theorem:2}
Let $F: \mathbb{R}^{2d} \rightarrow (-\infty, \infty]$ be a proper lower-semi-continuous function. Assume that  $\{z^k\}_{k \in \mathbb{N}}:=\{ (x^k, x^{k-1}) \}_{k \in \mathbb{N}}$ is a sequence generated by a general algorithm from $z^0:=(x^0,x^0)$,  for which the following three conditions are satisfied for any $k \in \mathbb{N}$.
\begin{description}
    \item [(H1)] For each $k \in \mathbb{N}$, there exists a positive `$a$' such that
	                 $$F(z^{k+1}) + a \, \|x^k-x^{k-1}\|^2 \leq F(z^{k}), \qquad \forall k \in \mathbb{N}. $$
	\item [(H2)] For each $k \in \mathbb{N}$, there exists a positive `$b$' such that for some $v^{k+1} \in \partial F(z^{k+1})$ we have
	  $$\|v^{k+1} \| \leq \frac{b}{2}\big(\|x^{k+1}-x^{k}\|+ \|x^k-x^{k-1}\| \big), \qquad \forall k \in \mathbb{N}.  $$
    \item [(H3)] There exists a subsequence $(z^{k_j})_{j\in \mathbb{N}}$ such that
    $z^{k_j} \rightarrow \tilde{z}$ and $ F(z^{k_j})\rightarrow F(\tilde{z})$.
\end{description}
Moreover,  if $F$ have the Kurdyka-{\L}ojasiewicz property at the limit point $\tilde{z}=(\tilde{x},\tilde{x} )$ specified in (H3). Then,  the sequence $\{x^k\}_{k \in \mathbb{N}}$ has finite length, i.e., $\sum_{k=1}^{\infty}\|x^k-x^{k-1}\| < \infty$, and converges to $\bar{x} =\tilde{x}$ as $k\rightarrow \infty$, where $(\bar{x}, \bar{x})$ is a critical point of $F$.
 \end{theorem}
In our paper, what we need is to verify that the conditions given in Theorem \ref{Theorem:2} are satisfied
for $F(x,y) = \Psi(x) + MD_h(y, x)$ and the sequence $(x^k, x^{k-1})_{k\in \mathbb{N}} \in \mathbb{R}^{2d}$ generated by the BPGe algorithm.

In order to guarantee the three conditions hold, we need another assumption. The first two requirements of the assumption were also required in \cite[see Assumption D(ii)]{bolte2018first}, and the third condition is easily verified.
\begin{assumption}\label{Assumption:C}
\begin{enumerate}[(i)]
\item $h$ is $\sigma$-strongly convex  on $\mathbb{R}^d$.\label{It2:1}
\item $\nabla h,\nabla f$ are Lipschitz continuous on any bounded subset of $\mathbb{R}^d$.\label{It2:2}
\item There exists a bounded $u$ such that $u \in \partial (\nabla h)$ on any bounded subset of $\mathbb{R}^d$.
\end{enumerate}
\end{assumption}
In fact, Assumption~\ref{Assumption:C}(\ref{It2:1}-\ref{It2:2}) can guarantee that Assumption~\ref{Assumption:B}(\ref{It1:2}-\ref{It1:3}) hold for the bounded sequence $\{x_k\}_{k \in \mathbb{N}}$.

The next task is to verify the three conditions one by one. Then, together with Theorem  \ref{Theorem:2}, we obtain the result that, under proper parameter selection, the whole sequence generated by BPGe converges to a critical point of the objective function.
\begin{theorem}\label{Theorem:3}
Suppose $\rho < \underline{\lambda}/\overline{\lambda}$. Let $\{x^k\}_{k\in \mathbb{N}} $ be a sequence generated from $x^0$ by BPGe.
If $F(x, y) = \Psi(x) + MD_h(y, x)$(where  $M \in ( \rho \, \underline{\lambda}^{-1}, \overline{\lambda}^{-1}]$ ) satisfies the Kurdyka--{\L}ojasiewicz property at some limit point $\tilde{z}=(\tilde{x},\tilde{x})\in \mathbb{R}^{2d}$ and Assumption~\ref{Assumption:C} holds, then
\begin{enumerate}[(i)]
\item The sequence $\{x^k\}_{k\in \mathbb{N}} $ has finite length, i.e. $\sum_{k=1}^{\infty}\|x^k-x^{k-1}\| <\infty$.
\item $x^k \rightarrow \tilde{x}$ as $k \rightarrow \infty$, and $\tilde{x}$ is a critical point of $\Psi$.
\end{enumerate}
\end{theorem}
\begin{proof}
We first verify the three conditions of the Theorem~\ref{Theorem:2} for function $H$ and BPGe algorithm.
\begin{enumerate}
\item[(H1)] According to  Assumption ~\ref{Assumption:C} , since $h$ is strongly convex,  assume that $h$ is $\sigma$-strongly convex, that is $D_h(x, y) \ge \frac{\sigma}{2}\|x-y\|^2$ for any $x, y \in \mathbb{R}^d$.
We denote  $ a =\frac{\sigma}{2}(M - \rho \, \underline{\lambda}^{-1} )$.
For any $k \in \mathbb{N}$,
\begin{align*}
 F(x^{k+1},x^k) + a \|x^k-x^{k-1}\|^2
 &\leq F(x^{k+1},x^k) + (M - \rho \underline{\lambda}^{-1} )D_h(x^{k-1},x^k) \\
 &\leq F(x^{k+1},x^k) + (M - \rho \lambda_k^{-1} )D_h(x^{k-1},x^k) \\
 & = H_{k+1,M} + (M - \rho \lambda_k^{-1} )D_h(x^{k-1},x^k) \\
 &\leq H_{k,M} + (M - \lambda_k^{-1})\, D_h(x^k, x^{k+1}) \\
 &\leq  H_{k,M}  = F(x^{k},x^{k-1}),
 \end{align*}
where the first inequality is based on the strongly convexity of $h$ , the second inequality is based on $ \underline{\lambda} \le \lambda_k$, the third and the last equality is from the definitions of $H_{k,M}$ and $F$, the fourth inequality is from Lemma \ref{Lemma:1}(ii), and the fifth inequality is according to the nonnegativeness of $(M - \lambda_k^{-1})\, D_h(x^k, x^{k+1})$. Thus (H1) is verified.
\item[(H2)] From the optimal condition (\ref{Eq:Alg2_2}), there exists
$ -\nabla f(y^k) + \lambda_k^{-1}(\nabla h(y^k) - \nabla h(x^{k+1})) \in \partial g(x^{k+1})$.
Due to Corollary \ref{Corollary:1}, $\{x^k\}_{k\in \mathbb{N}}$ generated by BPGe is bounded, and so also $\{y^k\}_{k\in \mathbb{N}}$ is bounded.
Thus, according to Assumption~\ref{Assumption:C}(iii), there exists a bounded $u_k \in \partial (\nabla h(x^k))$, and
\begin{small}
\begin{align*}
v_{k+1} = \left( \nabla f(x^{k+1})-\nabla f(y^{k}) -\lambda_{k}^{-1}(\nabla h(x^{k+1})- \nabla h(y^k))-M\langle u_k, x^{k+1} -x^k \rangle, \,M(\nabla h(x^{k})- \nabla h(x^{k+1})) \right),
\end{align*}
\end{small}
such that $v_{k+1} \in \partial F(x^{k+1},x^k)$. According to Assumption~\ref{Assumption:C}, there exist $L_f, \, L_h, \,  \delta$ such that for any $k \in \mathbb{N}$, $\|\nabla h(x^{k+1}) - \nabla h(y^{k})\| \leq L_h\|x^{k+1}-y^{k}\|$, $\|\nabla f(x^{k+1}) - \nabla f(y^{k})\| \leq L_f\|x^{k+1}-y^{k}\|$, $\|u_k\| \leq \delta$.
Hence
\begin{align}
 \|v_{k+1}\| &\leq \left(L_f+ \lambda_{k}^{-1} L_h \right) \| x^{k+1}-y^{k} \|+ M (\delta+L_h) \|x^{k+1}-x^{k} \| \nonumber \\
&\leq \left(L_f+ (\lambda_{k}^{-1}+M)L_h  + M \delta\right) \| x^{k+1}-x^{k} \| + \left(L_f+ \lambda_{k}^{-1} L_h \right) \| x^{k}-x^{k-1} \| \nonumber \\
&\leq  \left(L_f+ (\lambda_{k}^{-1}+M)L_h  + M \delta\right)(  \| x^{k+1}-x^{k} \|+ \| x^{k}-x^{k-1} \| ),
\end{align}
And so, (H2) is satisfied.
\item[(H3)] Condition (H3) naturally follows from Lemma \ref{Theorem:1}(ii).
\end{enumerate}
According to Theorem~\ref{Theorem:2}, combining the three conditions given in Theorem~\ref{Theorem:2} and KL property at $\tilde{z}$ could guarantee that conclusion (i) holds. Conclusion (ii) is followed by Theorem \ref{Theorem:3}(i).
Thus $\{x^k\}_{k\in \mathbb{N}}$ is a Cauchy sequence of $\mathbb{R}^d$ and converges to its limit point $\tilde{x}$. From Theorem~\ref{Theorem:2}   $\tilde{x}$ is the critical point.
\end{proof}

\section{Numerical Results}
\label{Sec:5}
In this section we perform several numerical tests in order to show the behaviour and the convergence speed up obtained when using the BPGe algorithm. We consider two important optimization problems in which the differentiable part of the objective \emph{does not admit} a  global Lipschitz continuous gradient: a convex Poisson linear inverse problem and a nonconvex quadratic inverse problem (and so the PG and PGe algorithms cannot be applied to these problems). It is important to remark that for cases where the differentiable part of the objective admits a  global Lipschitz continuous gradient the BPG and BPGe algorithms become the PG and PGe algorithms, respectively. That is, the BPG and BPGe methods can be applied but the performance in these cases it was already shown in \cite{wen2017linear}.

The main parameters in BPGe algorithm are the stepsizes $\lambda_k$ in Algorithm~1, and the parameter $\rho$ that gives the extrapolation coefficients $\beta_k$ in the line search method of Algorithm~2. In our tests we consider fixed stepsizes $\lambda_k = \lambda$. The influence of both parameters $\{ \lambda, \, \rho \}$ in order to fix suitable values is studied below in the tests.

All the numerical experiments have been performed in Matlab 2013a on a PC Intel(R) Xeon(R) CPU E5-2697 (2.6 GHz).

\subsection{Application to Poisson Linear Inverse Problems (PLIP)}

\begin{figure}[htbp]
  \centerline{\includegraphics[width=0.95\textwidth]{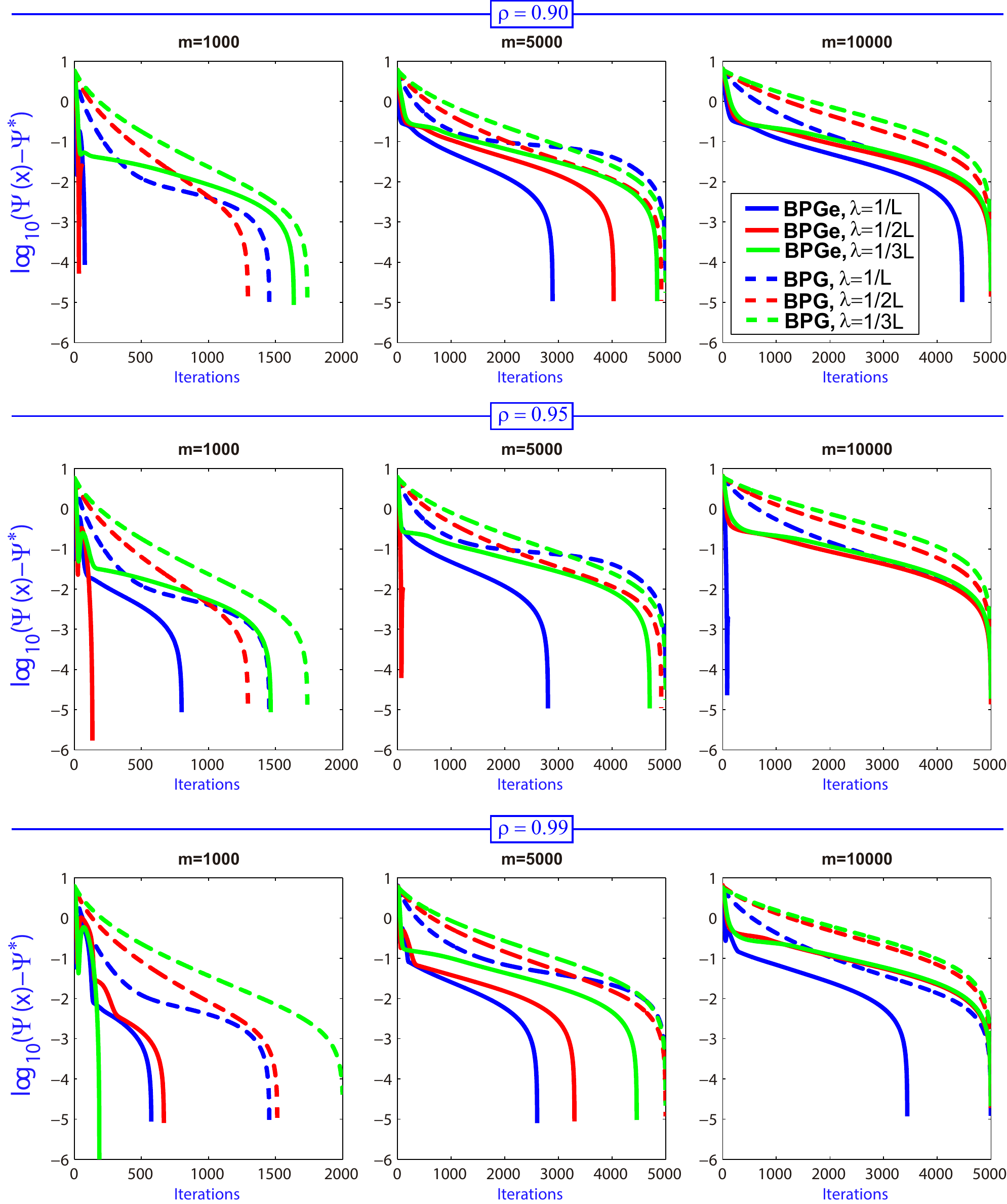}}
  \caption{Poisson Linear Inverse Problems tests (\emph{overdetermined} case $m > d$): evolution of the difference $\| \Psi(x_k)- \Psi(x^\ast)\|$ vs. iteration number, changing the parameters $\{ \lambda, \, \rho \}$ and for several problem sizes (measurements $m$) with fixed vector dimension $d= 100$.
  \label{Fig:1}}
\end{figure}

Poisson Linear Inverse Problems (that is, linear inverse problems in presence of Poisson noise) emerged in many fields, like astronomy, nuclear medicine (e.g., Positron Emission Tomography), inverse problems in fluorescence microscopy \cite{bauschke2016descent,bertero2009image,hohage2016}.
Therefore, the design of methods and estimators for such problems has been studied intensively over the last two decades (for a review, see \cite{bertero2009image,hohage2016}). Often these problems can be represented as a minimization problem like
\begin{align*}\label{PLIP}
\tag{PLIP} \min \big\{d(b,Ax) + \theta g(x): x \in \mathbb{R}^d_+ \big\}
\end{align*}
where $\theta>0$ is used to weigh matching the data fidelity criteria and its regularizer $g$, and $d(\cdot, \cdot)$ denotes a convex proximity measure between two vectors.

\begin{figure}[htb]
  \includegraphics[width=1.\textwidth]{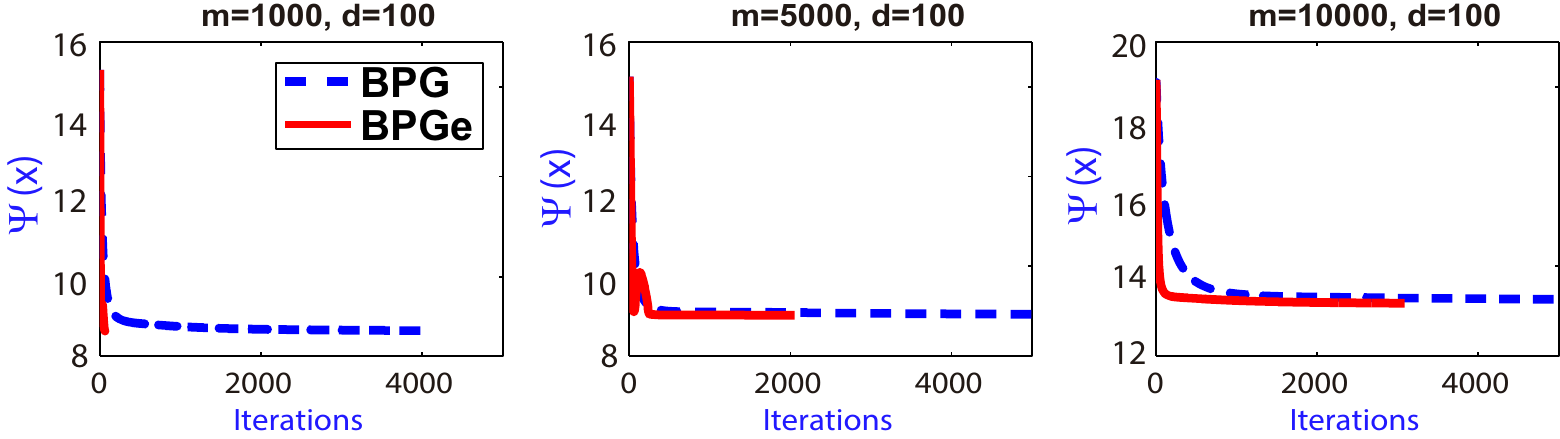}
  \caption{Poisson Linear Inverse Problems tests (\emph{overdetermined} case $m > d$): evolution of the objective function $\Psi(x_k)$ vs. iteration number, using the parameter values $\{ \lambda = 1/L, \, \rho = 0.99 \}$ and for several problem sizes (measurements $m$) with fixed vector dimension $d= 100$.
  \label{Fig:2a}}
\end{figure}

\begin{figure}[htb]
  \includegraphics[width=1.\textwidth]{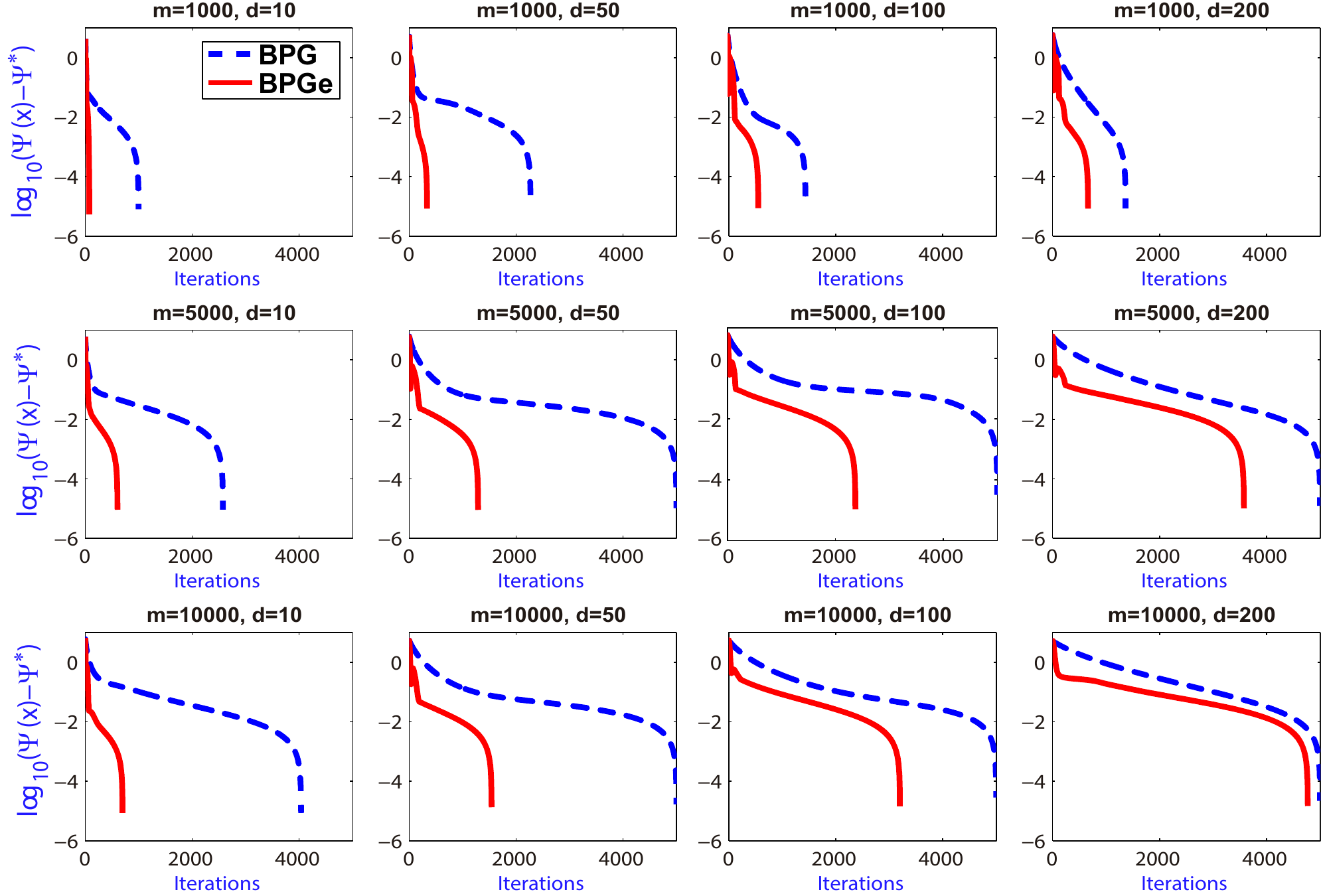}
  \caption{Poisson Linear Inverse Problems tests (\emph{overdetermined} case $m > d$): evolution of the difference $\| \Psi(x_k)- \Psi(x^\ast)\|$ vs. iteration number, using the parameter values $\{ \lambda = 1/L, \, \rho = 0.99 \}$ and for several problem sizes (measurements $m$ and vector dimensions $d$).
  \label{Fig:2}}
\end{figure}

A very well-known measure of proximity of two nonnegative vectors $Ax$ and $b$ is based on the Kullback-Liebler divergence: $$d(b,Ax) := \sum_{i=1}^m\bigg\{ b_i \log \frac{b_i}{(Ax)_i} + (Ax)_i - b_i\bigg\}.$$
which corresponds to noise of the negative Poisson log-likelihood function.
It is easy to find that $f:= d(b,Ax)$ has no globally Lipschitz continuous gradient  \cite{bauschke2016descent}, but satisfies $L$-smad condition with a kernel generating distance called Burg's entropy, denoted as $$h(x) = - \sum_{j=1}^d \log x_j, \text{dom}~h = \mathbb{R}_{+}^d,$$
and so now the Bregman distance is given by
$$
D_h(x,y) = \sum_{j=1}^d \bigg\{\frac{x_j}{y_j}- \log\bigg(\frac{x_j}{y_j}\bigg)-1\bigg\}.
$$

\begin{figure}[htbp]
  \centerline{\includegraphics[width=0.95\textwidth]{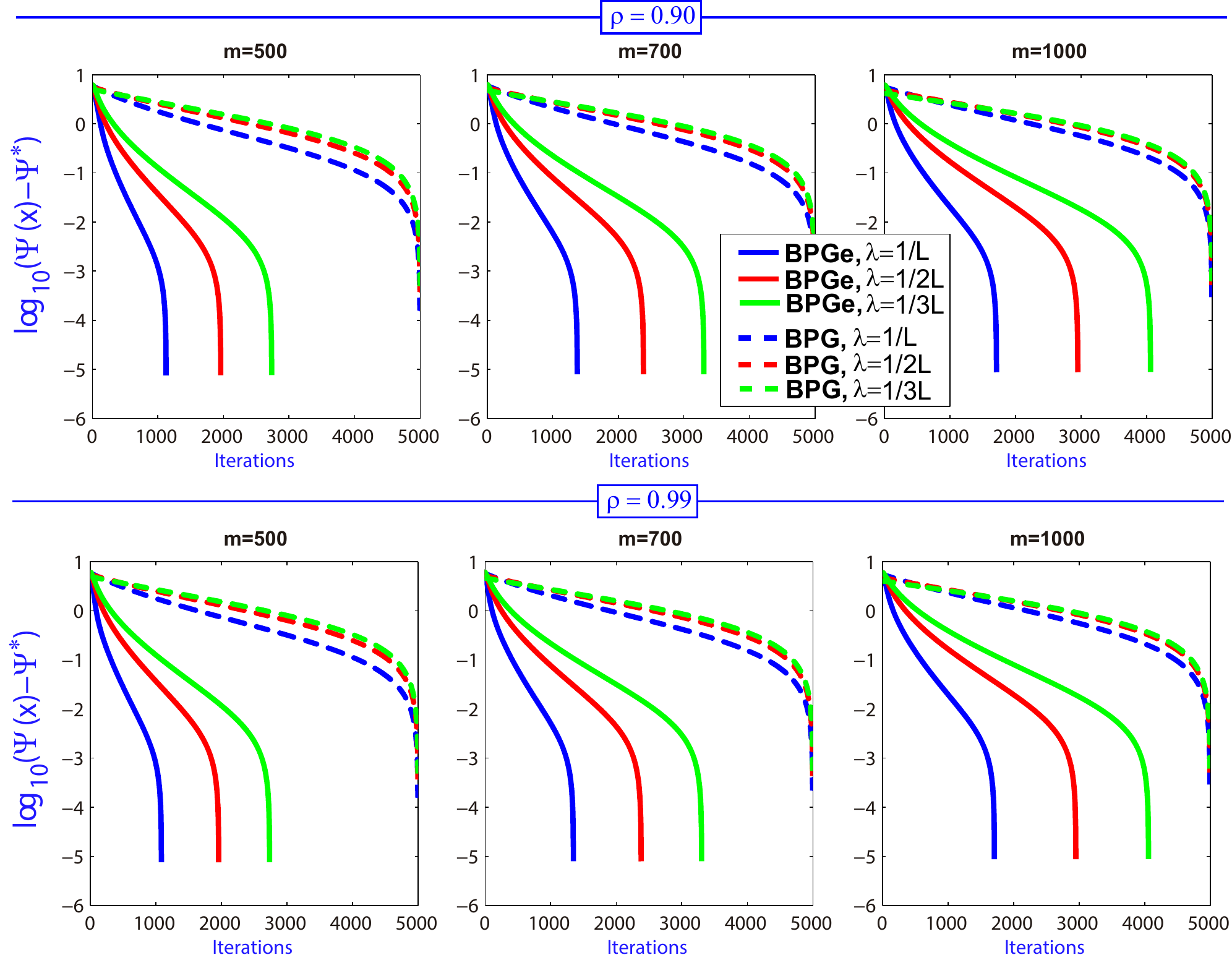}}
  \caption{Poisson Linear Inverse Problems tests (\emph{underdetermined} case $m < d$): evolution of the difference $\| \Psi(x_k)- \Psi(x^\ast)\|$ vs. iteration number, changing the parameters $\{ \lambda, \, \rho \}$ and for several problem sizes (measurements $m$) with fixed vector dimension $d= 5000$.
  \label{Fig:3}}
\end{figure}

\begin{figure}[htb]
  \centerline{\includegraphics[width=0.9\textwidth]{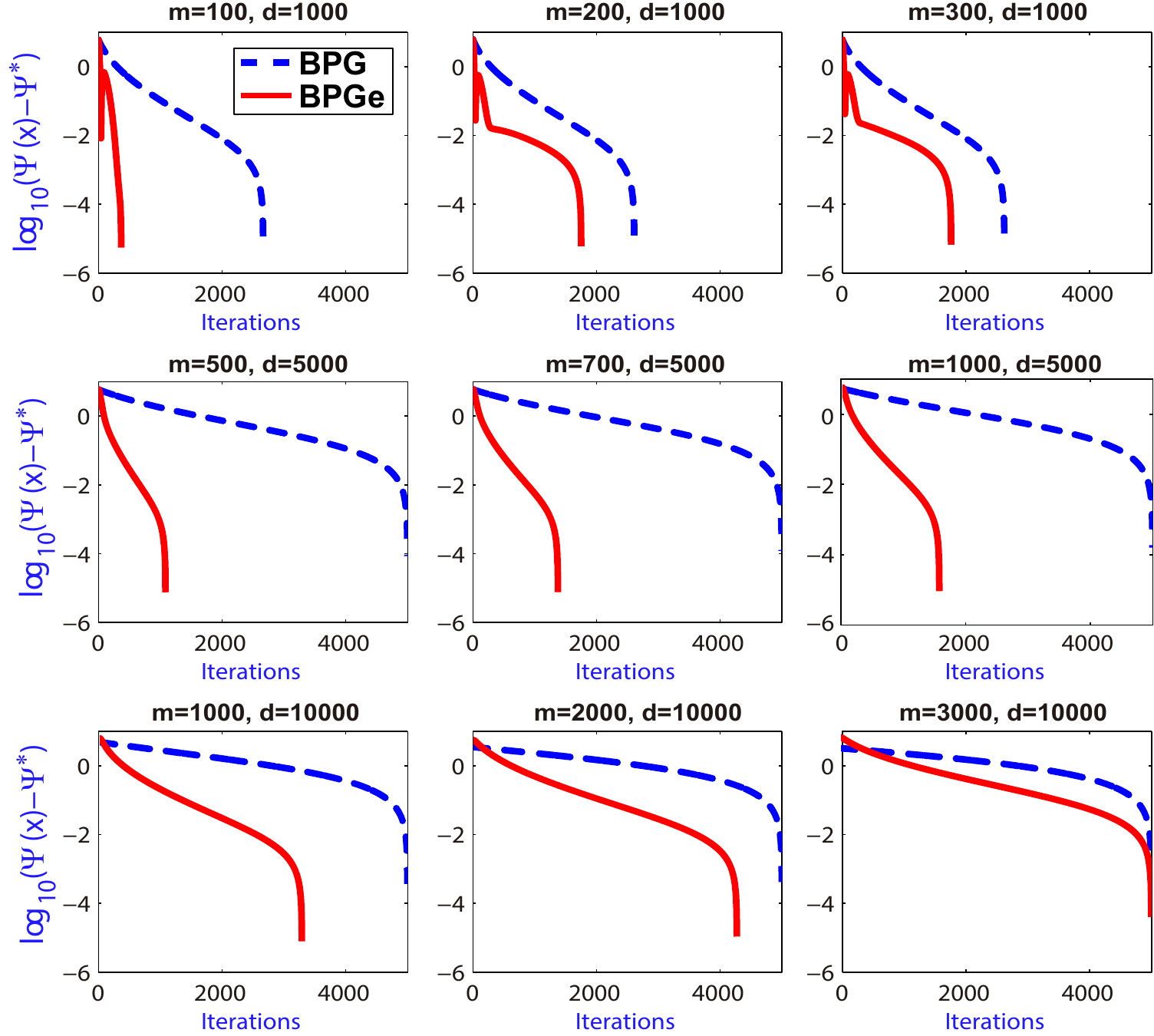}}
  \caption{Poisson Linear Inverse Problems tests (\emph{underdetermined} case $m < d$): evolution of the difference $\| \Psi(x_k)- \Psi(x^\ast)\|$ vs. iteration number, using the parameter values $\{ \lambda = 1/L, \, \rho = 0.99 \}$ and for several problem sizes (measurements $m$ and vector dimensions $d$).
  \label{Fig:4}}
\end{figure}

\begin{table*}[htbp]
\centering
  \small
\begin{center}{\emph{Overdetermined} case}\end{center}
  \hspace*{-0.7cm}{\tt \noindent\begin{tabular}{|r|r|rlrl|rlrl|}
    \cline{3-10}
    \multicolumn{2}{c|}{} & \multicolumn{4}{c|}{$\lambda = 1/L$}
       & \multicolumn{4}{c|}{$\lambda = 1/(3L)$} \\
    \hline
    m & d   & $T_{BPGe}$ & $\frac{T_{BPGe}}{T_{BPG}}$ & $N_{BPGe}$ & $\frac{NI_{BPGe}}{NI_{BPG}}  $& $T_{BPGe}$ & $\frac{T_{BPGe}}{T_{BPG}}$ & $N_{BPGe}$ & $\frac{NI_{BPGe}}{NI_{BPG}}$   \\
    \hline
    \hline
\rowcolor[gray]{0.9}
 1000 & 10  & 0.08  &   0.07 &  74  & 0.07        &    0.21  &  0.22  &   279  & 0.21  \\
      & 50  & 0.40  &   0.15 &  336 & 0.15        &    0.14  &  0.16  &   155  & 0.15${}^{a}$\\
      & 100 & 1.13  &   0.41 &  574 & 0.40        &    0.32  &  0.10  &   187  & 0.09${}^{a}$   \\
      & 200 & 1.68  &   0.63 &  665 & 0.49        &    0.44  &  0.07  &   226  & 0.07   \\
\rowcolor[gray]{0.9}
 5000 & 10  & 0.77  &   0.24 &  605  & 0.23        &    0.83  &  0.22  &   745  & 0.21  \\
      & 50  & 3.32  &   0.26 &  1291 & 0.26        &    4.16  &  0.34  &   1353 & 0.32 \\
      & 100 & 7.50  &   0.53 &  2602 & 0.52        &   13.97  &  0.96  &   4460 & 0.89   \\
      & 200 & 13.43 &   0.72 &  3577 & 0.72        &   20.26  &  1.12  &   5000 & 1.00     \\
\rowcolor[gray]{0.9}
 10000& 10  & 2.53   &   0.18 &   699& 0.17 &    0.50  &  0.03  &   141  & 0.03${}^{a}$\\
      & 50  & 6.68   &   0.33 &  1543& 0.31 &   15.36  &  0.68  &   3255 & 0.65 \\
      & 100 & 16.75  &   0.70 &  3441& 0.69 &   23.90  &  1.02  &   5000 & 1.00   \\
      & 200 & 30.32  &   0.99 &  4770& 0.95 &   30.20  &  1.05  &   5000 & 1.00 \\
    \hline
  \end{tabular}}

  \bigskip
  \begin{center}{\emph{Underdetermined} case}
  \end{center}
 \hspace*{-1.cm}{\tt \noindent\begin{tabular}{|r|r|rrrr|rrrr|}
    \cline{3-10}
    \multicolumn{2}{c|}{} & \multicolumn{4}{c|}{$\lambda = 1/L$}
       & \multicolumn{4}{c|}{$\lambda = 1/(3L)$} \\
    \hline
    m & d   & $T_{BPGe}$ & $\frac{T_{BPGe}}{T_{BPG}}$ & $N_{BPGe}$ & $\frac{NI_{BPGe}}{NI_{BPG}}  $& $T_{BPGe}$ & $\frac{T_{BPGe}}{T_{BPG}}$ & $N_{BPGe}$ & $\frac{NI_{BPGe}}{NI_{BPG}}$   \\
    \hline
    \hline
\rowcolor[gray]{0.9}
  100&   1000  &    0.60 &  0.15   & 369  &  0.14 &   2.19 &  0.25&   1314  & 0.26\\
  200&         &    5.03 &  0.89   & 1754 &  0.67 &   3.56 &  0.29&   1298  & 0.26\\
  300&         &    4.50 &  0.78   & 1760 &  0.67 &   2.81 &  0.26&   1315  & 0.26\\
\rowcolor[gray]{0.9}
  500 &  5000  &   9.17&   0.23 &  1085&   0.22  &   70.00 &  1.49 & 5000 &  1.00\\
  700 &        &  12.85&   0.28 &  1378&   0.28  &  115.16 &  1.34 & 5000 &  1.00\\
 1000 &        &  27.51&   0.32 &  1565&   0.31  &  345.13 &  1.18 & 5000 &  1.00\\
\rowcolor[gray]{0.9}
 1000 & 10000 &  210.06 &  0.66 &   3284 &  0.66 &   549.52 &  1.03 &   5000 &  1.00\\
 2000 &       &  643.71 &  0.89 &   4271 &  0.85 &   886.94 &  1.07 &   5000 &  1.00\\
 3000 &       &  967.90 &  1.02 &   5000 &  1.00 &   1084.82&  1.04 &   5000 &  1.00\\
    \hline
  \end{tabular}}
  \caption{Poisson Linear Inverse Problems tests: CPU-time and number of iterations for different cases of $m$ (number of data) and $d$ (dimension) for two different values of the $\lambda$ parameter for \emph{overdetermined} (top) and \emph{underdetermined} (bottom) cases. $T_{BPGe}$ and $T_{BPG}$ denote the CPU-time of BPGe and BPG algorithms, and $N_{BPGe}$ and $N_{BPG}$ the number of iterations to reach the \texttt{EXIT} criteria. Superscript --a-- points out discordant cases related with a fast linear convergence. \label{table1}}
  \end{table*}

Therefore, we have that
\begin{enumerate}[(i)]
	\item $(f,h)$ is $L$-smad, where $L \geq  \|b\|_1 $(according to Lemma~7 in \cite{bauschke2016descent}), and $f$ is $0$-relative weakly convex to $h$ since $f$ is convex;
	\item Assumptions \ref{Assumption:B}  and \ref{Assumption:A} hold, but Assumption \ref{Assumption:C} does not hold.
\end{enumerate}
So, from the convergence Section~\ref{Sec:4}, we can solve this problem using the BPGe algorithm and it is guaranteed that any limit point of the sequence generated by BPGe is a stationary point of the objective function $\Psi$.

An important point in any iterative method is to define suitable error control techniques. As discussed in Section~\ref{Sec:3},   \texttt{EXIT} conditions of the experiments are set when iterations exceed $5000$ times or $ \|x^k-x^{k-1}\|/\max\{1, \|x^k\| \} \leq 10^{-6}$ (as in \cite{wen2017linear}).

In the tests, the entries of $A \in \mathbb{R}_{+}^{m\times d}$ and $x\in \mathbb{R}_{+}^d$ are generated following independent uniform distribution over the interval $[0, 1]$. We consider the case $g(x) \equiv 0$, i.e., we solve the inverse problem without
regularization, so now the minimization problem is the standard Poisson
type maximum likelihood estimation problem (modulo change of sign to pass to a minimization problem).

As these methods (BPG and BPGe) can be applied to both, \emph{overdetermined} ($m>d$) and \emph{underdetermined} ($m<d$) problems, we have performed numerical tests on both cases. First, we present the results obtained in the \emph{overdetermined} case.
As commented before, the main parameters in BPGe algorithm are the stepsize $\lambda$ and the parameter $\rho$. In order to study briefly the most suitable set of parameters, we analyze the influence of both parameters $\{ \lambda, \, \rho \}$ in Figure~\ref{Fig:1}. In all the pictures we show the evolution of $\| \Psi(x_k)- \Psi(x^\ast)\|$ (being $x^\ast$ the approximate solution obtained at termination of each respective algorithm) with respect to the iteration number $k$. With this figure we can study the influence of the parameters with respect to the size of the problem (measurements $m$) with fixed dimension $d= 100$.
Globally, we observe that the value $\rho = 0.99$ has the best results, even if for some cases, the set of initial conditions gives rise to a very fast convergence (as in the cases of using $\lambda = 1/(2L)$ for $m=5000$ and $\rho = 0.95$, where we have a fast linear convergence instead of sublinear). Note that this kind of differences can be observed on other situations, but the average behaviour tells us that the best performance is when we take $\rho = 0.99$. On the other hand, similar comments can be said with respect to the stepsize parameter $\lambda$. The general situation also recommends us to take the highest value $\lambda = 1/L$ (also for both algorithms BPGe and BPG).

In Figure~\ref{Fig:2a}, now with the fixed parameter values $\{ \lambda = 1/L, \, \rho = 0.99 \}$ and for the \emph{overdetermined} ($m>d$) case, we show the evolution of the objective function $\Psi(x_k)$ vs. iteration number and for several problem sizes (measurements $m$) with fixed vector dimension $d= 100$. We observe that always the BPGe algorithm is much faster than the BPG one. In order to observe more clearly the faster convergence, we present in Figure~\ref{Fig:2} much more simulations but now showing the evolution of $\| \Psi(x_k)- \Psi(x^\ast)\|$. We note that the differences of both methods are bigger for low dimension $d$ problems, in fact for the most overdetermined problems $m\gg d$.

In the \emph{underdetermined} case we also analyze the influence of both parameters $\{ \lambda, \, \rho \}$ in Figure~\ref{Fig:3} with respect to the size of the problem (measurements $m$) with fixed dimension $d= 5000$.
Now, we observe that the value of the parameter $\rho$ seems to not affect too much on the global performance of the method, so we will take the value $\rho = 0.99$ when we fix the parameter. On the other hand, similar comments as in the \emph{overdetermined} case can be said with respect to the stepsize parameter $\lambda$. Now the behaviour is quite regular, and no cases of very fast convergence have been observed, and the fastest convergence is obtained for the highest value $\lambda = 1/L$ (also for both algorithms BPGe and BPG). Therefore, in the rest of tests on this paper we fix the parameter values $\{ \lambda = 1/L, \, \rho = 0.99 \}$.

In Figure~\ref{Fig:4}, now with the fixed parameter values $\{ \lambda = 1/L, \, \rho = 0.99 \}$  and for the \emph{underdetermined} ($m<d$) case, we observe that always the BPGe algorithm is much faster than the BPG one. But, similarly as in the \emph{overdetermined} case, the differences are bigger when we use the methods for larger ratios $d/m$, that is, for the most underdetermined problems $m\ll d$.

Finally, in Table~\ref{table1} we give the CPU-time and number of iterations for different sizes of problems (number of data $m$ and  dimension $d$) for two values of the $\lambda$ parameter ($\lambda=1/L$ and $1/(3L)$) for \emph{overdetermined} (top) and \emph{underdetermined} (bottom) cases. From the simulations we observe that when the problem has not a very big size (probably because in these other cases longer simulations are needed) the ratios among both methods provide an interesting speed-up, and in most cases the \texttt{EXIT} strategy stops the BPGe algorithm before the maximum number of iterations is reached. On the other hand, we observe that the CPU-time and iteration number ratios are quite similar, and so there are little differences between them. Note that the BPGe algorithm has an extra step, the line search method of Algorithm~2, but it increments quite a few the final CPU-time. On the table we have remarked three discordant cases (superscript --a--) related with a fast linear convergence, instead of sublinear. This is illustrated, for example, on the left bottom plot of Figure~\ref{Fig:1} ($\rho=0.99$,$m=1000$)  where the green curve, corresponding to $\lambda=1/(3L)$ converges faster than the other colours (as it also occurs in other plots of the same figure). Note that for an \emph{overdetermined} problem with random data some initial conditions and data may be led to a faster convergence. For the \emph{underdetermined} problem there is a regular behaviour in all the simulations.

Therefore, in the Poisson Linear Inverse Problems tests the BPGe algorithm presents a faster performance compared with the BPG algorithm, giving an interesting option for real problems.

\subsection{Application to Quadratic Inverse Problems}
 In the second test (taken from  \cite{bolte2018first}) we show that BPGe algorithm can deal with a nonconvex Quadratic Inverse Problem (QIP) in which the differentiable term has no globally gradient Lipschitz continuous property. This problem is a natural extension of the linear inverse problem, but now using quadratic measurements. It appears in many popular applications, such as signal recovery \cite{beck2013sparsity} and phase retrieve \cite{luke2017phase} from the knowledge of the amplitude of complex signals.

 A general description of the Quadratic Inverse Problem is to find the vector $x \in \mathbb{R}^d$ that solves the system
 \[
 x^T A_i x \simeq b_i, \quad i=1,\ldots,m
 \]
 being $\{A_i \in \mathbb{R}^{d \times d} \, | \, i=1,\ldots,m \}$ a set of symmetric matrices that describes the model, and $b = (b_1, \ldots, b_m) \in \mathbb{R}^m$ a vector of usually noisy measurements.

Following the formalism given in \cite[section 5.1]{bolte2018first}, this problem can be formulated as a nonconvex minimization problem as:
\begin{align*}
\tag{QIP} \min \left\{ \Psi(x):= \frac{1}{4} \sum_{i=1}^m (x^TA_ix-b_i)^2 +\theta g(x):x\in \mathbb{R}^d\right\},
\end{align*}
where $\theta>0$ is used to weigh matching the data fidelity criteria and its regularizer $g$. In our experiments, we take a convex $l_1$-norm regularization function $g(x)=\|x\|_1$. Note that the first function $f(x)$ is a nonconvex differentiable function but that does not admit a global Lipschitz continuous gradient.

The main quality of the BPG and BPGe algorithms (as noted to the BPG in \cite{bolte2018first}) is that these methods can solve the broad class of problems (QIP). To apply BPG and BPGe on the QIP model properly, we first need to identify a suitable function $h$ (Definition~\ref{Definition:Breg}). In \cite{bolte2018first}, a proper choice has been given as:
$$h(x) = \frac{1}{4}\|x\|_2^4 + \frac{1}{2}\|x\|_2^2,$$
and so now the Bregman distance is given by
$$
D_h(x,y) = \{h(x)-h(y)-(\| y \|^2y+y)^T(x-y)\}.
$$

When $L$ is chosen such that $L \geq \sum_{i=1}^m \left( 3\|A_i\|^2+ \|A_i\||b_i|\right)$
then by  \cite[Lemma 5.1]{bolte2018first}, \emph{$L$-smad} condition (Definition~\ref{Definition:Smad}) holds for the selected functions $f(x)$, $g(x)$ and $h(x)$. Besides, according to the same analysis in \cite[Lemma 5.1]{bolte2018first}, we could derive the relative weakly convex parameter as $\mu \geq \sum_{i=1}^m \|A_i\||b_i|$.
In conclusion, we have that:
\begin{enumerate}[(i)]
	\item $(f,h)$ is \emph{$L$-smad}, $f$ is $\mu$-relative weakly convex to $h$.
	\item Assumptions~\ref{Assumption:B} and \ref{Assumption:A} are easily verified.
	\item $f, g, D_h$ are all semi-algebraic, (see for example \cite{bolte2014proximal}). One can show inductively that $H_{M}(x,y)=\Psi(x)+ M D_h(x, y)$ is semi-algebraic, thus it has KL property (Definition~\ref{Definition:KL}) at any point $(x, x)$. Besides, we could verify that Assumption \ref{Assumption:C} holds.
\end{enumerate}
It means, from the convergence Section~\ref{Sec:4}, that the sequences generated by BPGe algorithm converge to a critical point of the objective function $\Psi$.

Here, we perform several numerical tests to compare the behaviour of the BPGe and BPG algorithms. As we did with the previous problem (PLIP), we have designed two main families of experiments, considering \emph{overdetermined} ($m > d$) and \emph{underdetermined} ($m< d$) cases.  To that goal we set different values of $m$ and $d$, and we generate $m$ random rank-1 matrices $A_i=a_i a_i^T$ in $\mathbb{R}^{d\times d}$, where the entries of the vectors $a_i$ are generated following independent Gaussian distributions with zero mean and unit variance. The accurate $x^{\ast}:= \arg\min\{\Psi(x): x \in \mathbb{R}^d \} $ is chosen as a sparse vector (the sparsity is $5\%$) and $b_i = x^TA_ix^{\ast}, \,  i = 1,\dots, m$. We set the weight parameter $\theta =1$ as default.

\begin{table*}[htb]
  \centering
  \small
  \hspace*{-1.2cm}{\tt \noindent\begin{tabular}{|r|r|rrrr|rrrr|}
    \cline{3-10}
    \multicolumn{2}{c|}{} & \multicolumn{4}{c|}{$\lambda = (1/L)$}
       & \multicolumn{4}{c|}{$\lambda = 1/(3L)$} \\
    \hline
    m & d   & $T_{BPGe}$ & $\frac{T_{BPGe}}{T_{BPG}}$ & $N_{BPGe}$ & $\frac{NI_{BPGe}}{NI_{BPG}}  $& $T_{BPGe}$ & $\frac{T_{BPGe}}{T_{BPG}}$ & $N_{BPGe}$ & $\frac{NI_{BPGe}}{NI_{BPG}}$   \\
    \hline
    \hline
\rowcolor[gray]{0.9}
10000 &    10&     0.29 &  0.53 &   146 &  0.35 & 0.48  & 0.28 & 248  & 0.20 \\
      &    50&     0.57 &  0.14 &   271 &  0.14 & 4.41  & 0.10 & 480  & 0.10 \\
      &   100&     1.16 &  0.10 &   339 &  0.08 & 8.73  & 0.19 & 655  & 0.13 \\
      &   200&    10.15 &  0.15 &   608 &  0.12 & 17.24 & 0.31 & 1668 & 0.33 \\
\rowcolor[gray]{0.9}
20000 &    10&     0.24 &  0.34 &   143 &  0.34 & 0.39  & 0.26 & 304  & 0.26 \\
      &    50&     4.09 &  0.14 &   266 &  0.14 & 6.80  & 0.16 & 465  & 0.09 \\
      &   100&     1.79 &  0.09 &   323 &  0.09 & 9.39  & 0.16 & 605  & 0.12 \\
      &   200&    66.97 &  0.18 &   602 &  0.12 & 40.74 & 0.28 & 1413 & 0.28 \\
\rowcolor[gray]{0.9}
30000 &    10&     0.40 &  0.44 &   145 &  0.35 & 3.22  & 0.27 & 231  & 0.20\\
      &    50&     1.48 &  0.15 &   261 &  0.15 & 10.42 & 0.10 & 472  & 0.10\\
      &   100&    32.79 &  0.09 &   331 &  0.09 & 15.06 & 0.12 & 594  & 0.12\\
      &   200&   153.17 &  0.12 &   554 &  0.11 & 487.62& 0.27 & 1341 & 0.27\\
    \hline
  \end{tabular}}
\caption{Quadratic Inverse Problems tests: CPU-time and number of iterations for different cases of $m$ (number of data) and $d$ (dimension) for two different values of the $\lambda$ parameter for the \emph{overdetermined} case. $T_{BPGe}$ and $T_{BPG}$ denote the CPU-time of BPGe and BPG algorithms, and $N_{BPGe}$ and $N_{BPG}$ the number of iterations to reach the \texttt{EXIT} criteria.  \label{table2}}
  \end{table*}

  \begin{figure*}[htbp]
  \centerline{\includegraphics[width=0.92\textwidth]{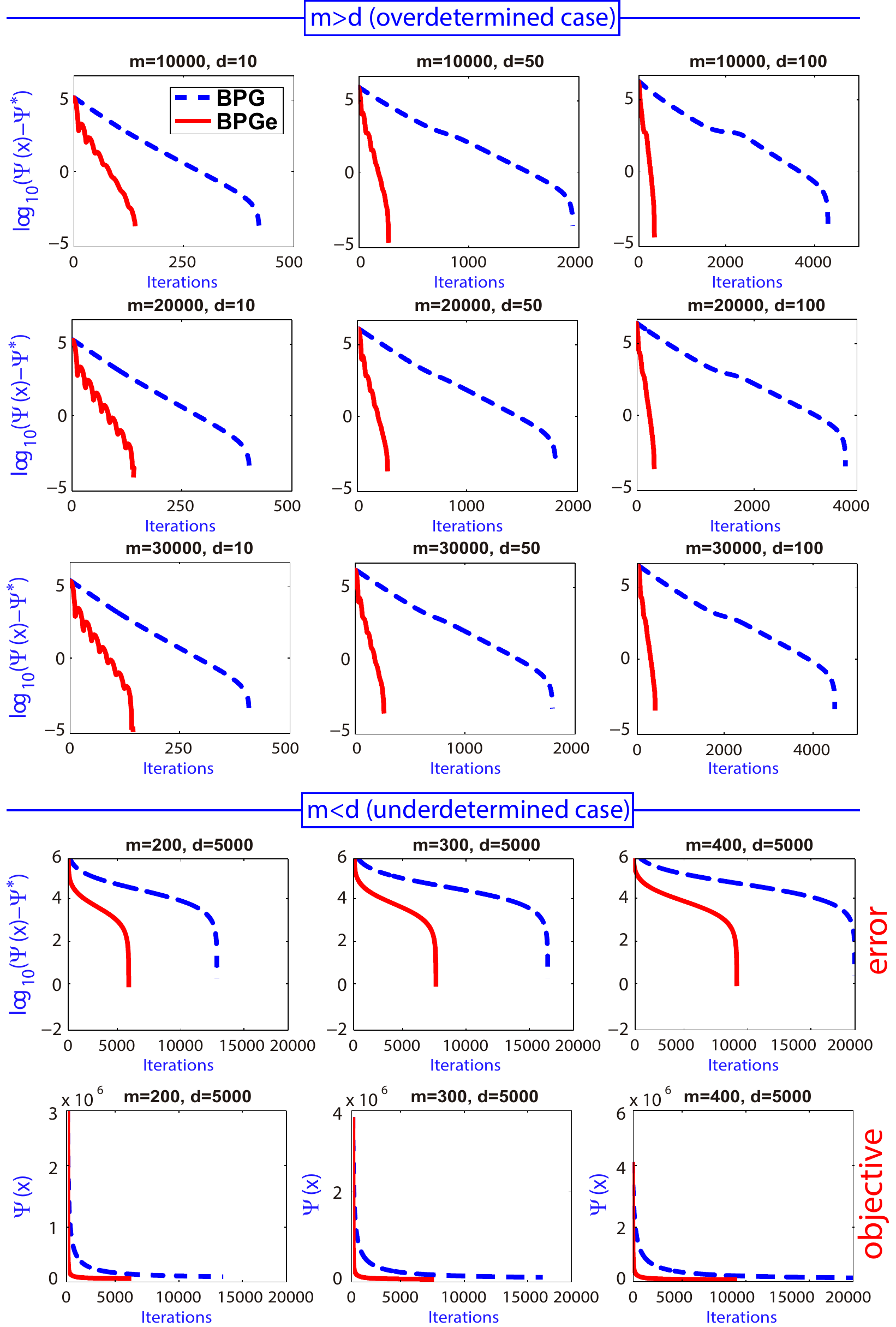}}
  \caption{Quadratic Inverse Problems tests (\emph{overdetermined} case $m > d$) and (\emph{underdetermined} case $m < d$): evolution of the difference $\| \Psi(x_k)- \Psi(x^\ast)\|$ vs. iteration number, using the parameter values $\{ \lambda = 1/L, \, \rho = 0.99 \}$ and for several problem sizes (measurements $m$ and vector dimensions $d$) and evolution of the objective function $\Psi(x_k)$.
  \label{Fig:6}}
\end{figure*}

As a first performance comparison, in Table~\ref{table2} we give the CPU-time and number of iterations for different sizes of problems (number of data $m$ and  dimension $d$) for two values of the $\lambda$ parameter ($\lambda=1/L$ and $1/(3L)$) for \emph{overdetermined} case. The values $T_{BPGe}$ and $T_{BPG}$ denote the CPU-time of BPGe and BPG algorithms, and $N_{BPGe}$ and $N_{BPG}$ the number of iterations to reach the \texttt{EXIT} criteria, respectively. From the simulations we observe that the ratios among both methods provide an interesting speed-up, and the \texttt{EXIT} strategy stops the BPGe algorithm before the maximum number of iterations ($k_{max}=5000$ in this case) is reached. On the other hand, we observe that the CPU-time and iteration number ratios are quite similar, and so there are little differences between them. Therefore, we note again that although the BPGe algorithm has an extra step (the line search method of Algorithm~2), it increments quite a few the final CPU-time. Also, from the data we observe that although the ratio for the BPGe and BPG algorithms for $\lambda = 1/(3L)$ is quite good, the option BPGe with $\lambda = 1/L$ performs many fewer iterations, and so it is the recommended option.

In Figure~\ref{Fig:6}, with the fixed parameter values $\{ \lambda = 1/L, \, \rho = 0.99 \}$ and for the \emph{overdetermined} ($m>d$)
and \emph{underdetermined} ($m<d$) cases, we show the evolution of $\| \Psi(x_k)- \Psi(x^\ast)\|$. In this problem we observe that
the performance of the accelerated BPGe algorithm for the \emph{overdetermined} case is quite good, giving a linear convergence. In the case of \emph{underdetermined} the behaviour seems to be sublinear, and it needs more iterations to reach the desired value (in this simulations $k_{max}=20000$). In both cases the BPGe algorithms performs much better than the BPG one.
For the \emph{underdetermined} case we also show the evolution of the objective function $\Psi(x_k)$ vs. iteration number to see that in this case the objective function takes large values, and therefore, when applying the \texttt{EXIT} strategy the required precision is obtained (a relative error $< 10^{-6}$) giving not too small absolute values.

Therefore, again in the Quadratic Inverse Problems tests the BPGe algorithm presents a faster performance compared with the BPG algorithm, giving an interesting option for real problems.

\section{Conclusions}

This work have joined two powerful methods to solve large-scale minimization problems and we proposed a new accelerated Bregman proximal gradient algorithm (BPGe) useful for nonconvex and nonsmooth minimization problems. On one hand, we have taken the BPG algorithm~\cite{bauschke2016descent} able to deal with non-globally Lipschitz continuous gradient problems. Firstly defined for the convex case~\cite{bauschke2016descent} and later extended to the nonconvex case by \cite{bolte2018first}. And on the other hand, the accelerated extrapolation algorithm (used for instance in the PG algorithm \cite{wen2017linear}). The use of the Bregman distance paradigm permits to enlarge the number of problems to work with, because we do not need the assumption of global Lipschitz gradient continuity. And with the extrapolation technique the convergence of the method is accelerated.

In this paper we have studied the convergence of the new method, and we have proven that any limit point of the sequence generated by BPGe algorithm is a stationary point of the problem by choosing parameters properly. Besides, assuming Kurdyka-{\L}ojasiewicz property, we have proven the whole sequences generated by BPGe converges to a stationary point.

Finally, we have applied it to two important practical problems that arise in many fundamental applications (and that not satisfy global Lipschitz gradient continuity assumption): Poisson linear inverse problems and quadratic inverse problems, for both, \emph{overdetermined} and \emph{underdetermined} cases.  In these tests  the BPGe algorithm have shown faster convergence results than the BPG algorithm, and so the new BPGe algorithm seems to be an interesting methodology.


\bibliographystyle{num}

\begin{thebibliography}{10}

\bibitem{Bertsekas2009}
D.~Bertsekas, Convex optimization theory, Athena Scientific, 2009.

\bibitem{borwein2010convex}
J.~Borwein, A.~S. Lewis, Convex analysis and nonlinear optimization: theory and
  examples, Springer Science \& Business Media, 2010.

\bibitem{bauschke2016descent}
H.~H. Bauschke, J.~Bolte, M.~Teboulle, A descent lemma beyond {L}ipschitz
  gradient continuity: first-order methods revisited and applications,
  Mathematics of Operations Research 42~(2) (2017) 330--348.

\bibitem{bolte2018first}
J.~Bolte, S.~Sabach, M.~Teboulle, Y.~Vaisbourd, First order methods beyond
  convexity and {L}ipschitz gradient continuity with applications to quadratic
  inverse problems, SIAM Journal on Optimization 28~(3) (2018) 2131--2151.

\bibitem{beck2009fast}
A.~Beck, M.~Teboulle, A fast iterative shrinkage-thresholding algorithm for
  linear inverse problems, SIAM journal on imaging sciences 2~(1) (2009)
  183--202.

\bibitem{Nesterov1983}
Y.~Nesterov, A method of solving a convex programming problem with convergence
  rate $o(1/k^2)$, Soviet Math. Dokl. 27 (1983) 372--376.

\bibitem{Nesterov2007}
Y.~Nesterov, Dual extrapolation and its applications to solving variational
  inequalities and related problems, Mathematical Programming 109~(2) (2007)
  319--344.

\bibitem{wen2017linear}
B.~Wen, X.~Chen, T.~K. Pong, Linear convergence of proximal gradient algorithm
  with extrapolation for a class of nonconvex nonsmooth minimization problems,
  SIAM Journal on Optimization 27~(1) (2017) 124--145.

\bibitem{donoho2006compressed}
D.~L. Donoho, Compressed sensing, IEEE Transactions on information theory
  52~(4) (2006) 1289--1306.

\bibitem{beck2013sparsity}
A.~Beck, Y.~C. Eldar, Sparsity constrained nonlinear optimization: Optimality
  conditions and algorithms, SIAM Journal on Optimization 23~(3) (2013)
  1480--1509.

\bibitem{luke2017phase}
D.~R. Luke, Phase retrieval, what’s new, SIAG/OPT Views and News 25~(1)
  (2017) 1--5.

\bibitem{parikh2014proximal}
N.~Parikh, S.~Boyd, et~al., Proximal algorithms, Foundations and
  Trends{\textregistered} in Optimization 1~(3) (2014) 127--239.

\bibitem{schmidt2011convergence}
M.~Schmidt, N.~L. Roux, F.~R. Bach, Convergence rates of inexact
  proximal-gradient methods for convex optimization, in: Advances in neural
  information processing systems, 2011, pp. 1458--1466.

\bibitem{jiang2012inexact}
K.~Jiang, D.~Sun, K.-C. Toh, An inexact accelerated proximal gradient method
  for large scale linearly constrained convex {SDP}, SIAM Journal on
  Optimization 22~(3) (2012) 1042--1064.

\bibitem{xiao2014proximal}
L.~Xiao, T.~Zhang, A proximal stochastic gradient method with progressive
  variance reduction, SIAM Journal on Optimization 24~(4) (2014) 2057--2075.

\bibitem{nitanda2014stochastic}
A.~Nitanda, Stochastic proximal gradient descent with acceleration techniques,
  in: Advances in Neural Information Processing Systems, 2014, pp. 1574--1582.

\bibitem{chen2012fast}
A.~I. Chen, A.~Ozdaglar, A fast distributed proximal-gradient method, in:
  Communication, Control, and Computing (Allerton), 2012 50th Annual Allerton
  Conference on, IEEE, 2012, pp. 601--608.

\bibitem{vanli2018global}
N.~D. Vanli, M.~Gurbuzbalaban, A.~Ozdaglar, Global convergence rate of proximal
  incremental aggregated gradient methods, SIAM Journal on Optimization 28~(2)
  (2018) 1282--1300.

\bibitem{toh2010accelerated}
K.-C. Toh, S.~Yun, An accelerated proximal gradient algorithm for nuclear norm
  regularized linear least squares problems, Pacific Journal of optimization
  6~(15) (2010) 615--640.

\bibitem{ghadimi2016accelerated}
S.~Ghadimi, G.~Lan, Accelerated gradient methods for nonconvex nonlinear and
  stochastic programming, Mathematical Programming 156~(1-2) (2016) 59--99.

\bibitem{li2015accelerated}
H.~Li, Z.~Lin, Accelerated proximal gradient methods for nonconvex programming,
  in: Advances in neural information processing systems, 2015, pp. 379--387.

\bibitem{carmon2018accelerated}
Y.~Carmon, J.~C. Duchi, O.~Hinder, A.~Sidford, Accelerated methods for
  nonconvex optimization, SIAM Journal on Optimization 28~(2) (2018)
  1751--1772.

\bibitem{bertero2009image}
M.~Bertero, P.~Boccacci, G.~Desider{\`a}, G.~Vicidomini, Image deblurring with
  {P}oisson data: from cells to galaxies, Inverse Problems 25~(12) (2009)
  123006.

\bibitem{boct2016inertial}
R.~I. Bo{\c{t}}, E.~R. Csetnek, S.~C. L{\'a}szl{\'o}, An inertial
  forward--backward algorithm for the minimization of the sum of two nonconvex
  functions, EURO Journal on Computational Optimization 4~(1) (2016) 3--25.

\bibitem{hanzely2018accelerated}
F.~Hanzely, P.~Richtarik, L.~Xiao, Accelerated bregman proximal gradient
  methods for relatively smooth convex optimization, arXiv preprint
  arXiv:1808.03045.

\bibitem{rockafellar2015convex}
R.~T. Rockafellar, Convex analysis, Princeton university press, 2015.

\bibitem{rockafellar2009variational}
R.~T. Rockafellar, R.~J.-B. Wets, Variational analysis, Vol. 317, Springer
  Science \& Business Media, 2009.

\bibitem{lu2018relatively}
H.~Lu, R.~M. Freund, Y.~Nesterov, Relatively smooth convex optimization by
  first-order methods, and applications, SIAM Journal on Optimization 28~(1)
  (2018) 333--354.

\bibitem{bregman1967relaxation}
L.~M. Bregman, The relaxation method of finding the common point of convex sets
  and its application to the solution of problems in convex programming, USSR
  Computational Mathematics and Mathematical Physics 7~(3) (1967) 200--217.

\bibitem{chen1993convergence}
G.~Chen, M.~Teboulle, Convergence analysis of a proximal-like minimization
  algorithm using {B}regman functions, SIAM Journal on Optimization 3~(3)
  (1993) 538--543.

\bibitem{teboulle2018simplified}
M.~Teboulle, A simplified view of first order methods for optimization,
  Mathematical Programming (2018) 1--30.

\bibitem{davis2018stochastic}
D.~Davis, D.~Drusvyatskiy, K.~J. MacPhee, Stochastic model-based minimization
  under high-order growth, arXiv preprint arXiv:1807.00255.

\bibitem{nurminskii1973quasigradient}
E.~Nurminskii, The quasigradient method for the solving of the nonlinear
  programming problems, Cybernetics 9~(1) (1973) 145--150.

\bibitem{bolte2014proximal}
J.~Bolte, S.~Sabach, M.~Teboulle, Proximal alternating linearized minimization
  for nonconvex and nonsmooth problems, Mathematical Programming 146~(1-2)
  (2014) 459--494.

\bibitem{li2017calculus}
G.~Li, T.~K. Pong, Calculus of the exponent of {K}urdyka--{{\L}}ojasiewicz
  inequality and its applications to linear convergence of first-order methods,
  Foundations of Computational Mathematics (2017) 1--34.

\bibitem{attouch2013convergence}
H.~Attouch, J.~Bolte, B.~F. Svaiter, Convergence of descent methods for
  semi-algebraic and tame problems: proximal algorithms, forward--backward
  splitting, and regularized {Ga}uss--{S}eidel methods, Mathematical
  Programming 137~(1-2) (2013) 91--129.

\bibitem{ochs2014ipiano}
P.~Ochs, Y.~Chen, T.~Brox, T.~Pock, i{P}iano: Inertial proximal algorithm for
  nonconvex optimization, SIAM Journal on Imaging Sciences 7~(2) (2014)
  1388--1419.

\bibitem{pock2016inertial}
T.~Pock, S.~Sabach, Inertial proximal alternating linearized minimization
  {(iPALM)} for nonconvex and nonsmooth problems, SIAM Journal on Imaging
  Sciences 9~(4) (2016) 1756--1787.

\bibitem{hohage2016}
T.~Hohage, F.~Werner, Inverse problems with {P}oisson data: statistical
  regularization theory, applications and algorithms, Inverse Problems 32~(9)
  (2016) 093001.

\end{thebibliography}

\end{document}